\documentclass{amsart}

\usepackage[all]{xypic}
\usepackage[centertags]{amsmath}
\usepackage{amsfonts}
\usepackage{amscd}
\usepackage{amssymb}
\usepackage{amsthm}
\usepackage{newlfont}
\usepackage{amsxtra}
 \newtheorem{thm}{Theorem}[section]
 \newtheorem{cor}[thm]{Corollary}
 
 \newtheorem{lem}[thm]{Lemma}
 \newtheorem{prop}[thm]{Proposition}
 \theoremstyle{definition}
 \newtheorem{defn}[thm]{Definition}
 \theoremstyle{remark}
 \newtheorem{rem}[thm]{Remark}
 \theoremstyle{remark}
 \newtheorem{example}[thm]{Example}
 \theoremstyle{definition}
 \newtheorem{notn}[thm]{Notation}
 \numberwithin{equation}{section}

 \newcommand{\an}{\mathrm{an}}

 \newcommand{\Ver}{\mathrm{Ver}}
 
 \newcommand{\Spec}{\mathrm{Spec}}
 \newcommand{\Frob}{\mathrm{Frob}}
 \newcommand{\Spf}{\mathrm{Spf}}

 \newcommand{\Aut}{\mathrm{Aut}}
 
 \newcommand{\Pic}{\mathrm{Pic}}
 
 \newcommand{\ord}{\mathrm{ord}}
 
 \newcommand{\Gr}{\mathrm{Gr}}

 \newcommand{\GL}{\mathrm{GL}}
 \newcommand{\PGL}{\mathrm{PGL}}
 \newcommand{\SL}{\mathrm{SL}}

 \newcommand{\un}{\mathrm{ur}}

 \newcommand{\loc}{\mathrm{loc}}
 
 \newcommand{\tor}{\mathrm{tor}}
 \newcommand{\Stab}{\mathrm{Stab}}

 \newcommand{\Ed}{\mathrm{Ed}}

 \newcommand{\Id}{\mathrm{Id}}
 \newcommand{\Tr}{\mathrm{Tr}}

 \newcommand{\Fr}{\mathrm{Fr}}

 \newcommand{\inv}{\mathrm{inv}}
 \renewcommand{\mod}{\mathrm{mod}}
 \newcommand{\Nr}{\mathrm{Nr}}
 \newcommand{\Odd}{\mathrm{Odd}}

 \newcommand{\fp}{\mathfrak p}
 \newcommand{\fr}{\mathfrak r}

 \newcommand{\fD}{\mathfrak D}

 \newcommand{\cO}{\mathcal{O}}
 \newcommand{\cZ}{\mathcal{Z}}

 \newcommand{\cC}{\mathcal{C}}
 \newcommand{\cA}{\mathcal{A}}
 
 \newcommand{\cG}{\mathcal{G}}

 \renewcommand{\cD}{\mathcal{D}}
 \newcommand{\cE}{\mathcal{E}}

 \newcommand{\cI}{\mathcal{I}}
 
 \renewcommand{\cH}{\mathcal{H}}
 \newcommand{\cT}{\mathcal{T}}

 \newcommand{\R}{\mathbb{R}}
 \newcommand{\C}{\mathbb{C}}
 \newcommand{\E}{\mathbb{E}}
 \newcommand{\F}{\mathbb{F}}
 \newcommand{\M}{\mathbb{M}}
 \newcommand{\Q}{\mathbb{Q}}
 \newcommand{\Z}{\mathbb{Z}}
 \newcommand{\A}{\mathbb{A}}
 
 \newcommand{\N}{\mathbb{N}}
 \renewcommand{\P}{\mathbb{P}}

 \newcommand{\Ell}{\mathcal{E}\ell\ell}

 \newcommand{\eps}{\varepsilon}
 \newcommand{\To}{\longrightarrow}
 
 \newcommand{\bs}{\setminus}
 \newcommand{\bD}{\bar{D}}
 \newcommand{\tF}{\widetilde{F}}
 \newcommand{\tP}{\widetilde{\Pi}}

 \newcommand{\Fi}{F_\infty}

 \newcommand{\G}{\Gamma}
 
 \newcommand{\La}{\Lambda}
 \newcommand{\la}{\lambda}

 \newcommand{\twist}[1]{{^\tau}\!#1}

\begin{document}

\title[Properties of modular curves of $\cD$-elliptic sheaves]
{Local diophantine properties of \\ modular curves of $\cD$-elliptic sheaves}

\author{Mihran Papikian}

\address{Department of Mathematics, Pennsylvania State University, University Park, PA 16802}

\email{papikian@math.psu.edu}

\thanks{The author was supported in part by NSF grant DMS-0801208 and Humboldt Research Fellowship.}



\begin{abstract}
We study the existence of rational points on modular curves of
$\cD$-elliptic sheaves over local fields and the structure of
special fibres of these curves. We discuss some applications which
include finding presentations for arithmetic groups arising from
quaternion algebras, finding the equations of modular curves of
$\cD$-elliptic sheaves, and constructing curves violating the Hasse
principle.
\end{abstract}


\maketitle


\section{Introduction}\label{SecIntr}

The purpose of this paper is to discuss the function field analogues
of certain problems about Shimura curves over local fields. To
describe our results we first need to introduce some notation.

Let $C:=\P^1_{\F_q}$ be the projective line over the finite field
$\F_q$. Denote by $F=\F_q(T)$ the field of rational functions on
$C$. The set of closed points on $C$ (equivalently, places of $F$)
is denoted by $|C|$. For each $v\in |C|$, we denote by $\cO_v$ and
$F_v$ the completions of $\cO_{C,v}$ and $F$ at $v$, respectively.
Let $A:=\F_q[T]$. This is the subring of $F$ consisting of functions
which are regular away from the place generated by $1/T$ in
$\F_q[1/T]$. The place generated by $1/T$ will be denoted by
$\infty$ and called the \textit{place at infinity}; it will play a
role similar to the archimedean place for $\Q$. The places in
$|C|-\infty$ are the \textit{finite places}.

Let $D$ be a quaternion division algebra over $F$, which is split at
$\infty$, i.e., $D\otimes_F\Fi\cong \M_2(\Fi)$. Let $R$ be the
finite set of places where $D$ ramifies; see $\S$\ref{Sec3} for the
terminology and basic properties. Denote by $D^\times$ the
multiplicative group of $D$. Fix a maximal $A$-order $\La$ in $D$.
Since $D$ is split at $\infty$, it satisfies the so-called
\textit{Eichler condition} relative to $A$, which implies that, up
to conjugation, $\La$ is the unique maximal $A$-order in $D$, cf.
\cite[Cor. III. 5.7]{Vigneras}. Let
$$
\G^R:=\La^\times=\{\la\in \La\ |\ \Nr(\la)\in \F_q^\times\}
$$
be the group of units of $\La$. Via an isomorphism
$D^\times(\Fi)\cong \GL_2(\Fi)$, the group $\G^R$ can be considered
as a discrete subgroup of $\GL_2(\Fi)$. There are two analogues of
the Poincar\'e upper half-plane in this setting. One is the
\textit{Bruhat-Tits tree} $\cT$ of $\PGL_2(\Fi)$ and the other is
\textit{Drinfeld half-plane}
$\Omega:=\P^{1,\an}_{\Fi}-\P^{1,\an}_{\Fi}(\Fi)$, where
$\P^{1,\an}_{\Fi}$ is the rigid-analytic space associated to the
projective line over $\Fi$. These two versions of the upper
half-plane are related to each other: $\cT$ is the dual graph of an
analytic reduction of $\Omega$, cf. \cite{vdPut}. As a subgroup of
$\GL_2(\Fi)$, $\G^R$ acts naturally on both $\cT$ and $\Omega$
(these actions are compatible with respect to the reduction map).
The quotient $\G^R\bs \Omega$ is a one-dimensional, connected,
smooth analytic space over $\Fi$. It is the rigid-analytic space
associated to a smooth, projective curve $X^R$ over $\Fi$ \cite[Thm.
3.3]{vdPut}. In fact, $X^R$ is a moduli scheme of certain objects,
called \textit{$\cD$-elliptic sheaves}, so it has a model over $F$;
cf. $\S$\ref{ssDES}. The curve $X^R$ is the function field analogue
of a Shimura curve parametrizing abelian surfaces with
multiplication by a maximal order in an indefinite division
quaternion algebra over $\Q$.

\vspace{0.1in}

Now we describe the main results of the paper. Let $K$ be a finite
extension of $F_v$. We determine whether $X^R$ has $K$-rational
points. This problem naturally breaks into three cases, which need
to be examined separately: $v\in |C|-R-\infty$, $v\in R$ and
$v=\infty$. The corresponding theorems are Theorem \ref{thm2.1},
\ref{thm3.1} and \ref{thmLPinf}. These results are the function
field analogues of the results of Jordan and Livn\'e \cite{JL} and
Shimura \cite{Shimura}.

Next, we study the quotient graphs $\G^R\bs \cT$, which are the
analogues of fundamental domains of Shimura curves in the Poincar\'e
upper half-plane. The study of such domains over $\C$ is a classical
problem which can be traced back to the 19th century. Nevertheless,
determining explicitly such a domain for a given arithmetic group is
a computationally difficult problem; in fact, a large portion of the
recent book \cite{AB} is devoted to this problem (see also
\cite{KV}). We give a description of $\G^R\bs \cT$ in Corollary
\ref{cor5.4} and Theorem \ref{PropTV}.

Finally, we discuss some applications of our results. We give an
upper-bound on the number of generators of $\G^R$, determine the
cases when $\G^R$ is generated by torsion elements and find a
presentation for $\G^R$ in those cases (see Theorems \ref{thmGTI}
and \ref{thm-tree}). In $\S$\ref{SecExplU}, we find explicitly the
torsion units which generate $\G^R$ in terms of a basis of $D$. We
also determine in some special cases the equation defining $X^R$ as
a curve over $F$ (Theorem \ref{thmEquations}); as far as I am aware,
these are the first known examples of such equations. In
$\S$\ref{SHasse}, we use the knowledge of local points on $X^R$ to
show that there exist quadratic extensions of $F$ over which $X^R$
violates the so-called Hasse principle (Theorem \ref{thmHP}).


\section{Preliminaries} The purpose of this section is to fix the
notation and terminology, and recall some basic facts which will be
used later in the paper.

\subsection{Notation} Besides the notation in the introduction, the
following will be used throughout the paper. Let $v\in |C|$. The
residue field of $\cO_v$ is denoted by $\F_v$, the cardinality of
$\F_v$ is denoted by $q_v$, the degree $m$ extension of $\F_v$ is
denoted by $\F_v^{(m)}$, and $\deg(v):=\dim_{\F_q} (\F_v)$. We
assume that the valuation $\ord_v:F_v\to\Z$ is normalized by
$\ord_v(\varpi_v)=1$, where $\varpi_v$ is a uniformizer of $\cO_v$.
Denote the adele ring of $F$ by $\A=\prod'_{v\in |C|} F_v$, and for
a finite set $S\subset |C|$, let $\A^S:=\prod'_{v\in |C|-S}F_v$ be
the adele ring outside of $S$. For each $v\in |C|-\infty$, let
$\fp_v\lhd A$ be the corresponding prime ideal of $A$ and $\wp_v\in
A$ be the monic generator of $\fp_v$. For $a\in A$, let $\deg(a)$ be
the degree of $a$ as a polynomial in $T$. Note that
$\deg(\wp_v)=\deg(v)$.

For a ring $H$ with a unit element, we denote by $H^\times$ the
group of all invertible elements of $H$.

For $S\subset |C|$, put
$$
\Odd(S)=\left\{
         \begin{array}{ll}
           1, & \hbox{if all places in $S$ have odd degrees;} \\
           0, & \hbox{otherwise.}
         \end{array}
       \right.
$$


\subsection{Quaternion algebras}\label{Sec3} Let $D$ be a \textit{quaternion
algebra} over $F$, i.e., a $4$-dimensional $F$-algebra with center
$F$ which does not possess non-trivial two-sided ideals. If $L$ is a
field containing $F$, then $D\otimes_F L$ is a quaternion algebra
over $L$, cf. \cite[p. 4]{Vigneras}. By Wedderburn's theorem
\cite[(7.4)]{Reiner}, $D\otimes_F L$ is either a division algebra or
is isomorphic to the matrix algebra $\M_2(L)$. In the second case,
we say that $L$ \textit{splits} $D$, or is a \textit{splitting
field} for $D$, cf. \cite[p. 96]{Reiner}. We say that $D$
\textit{splits} (resp. \textit{ramifies}) at $v\in |C|$ if $F_v$ is
a splitting field for $D$ (resp. is not a splitting field). Let
$R\subset |C|$ be the set of places where $D$ ramifies. It is known
that $R$ is a finite set and has even cardinality, and conversely,
for any choice of a finite set $R\subset |C|$ of even cardinality
there is a unique, up to isomorphism, quaternion algebra ramified
exactly at the places in $R$; see \cite[p. 74]{Vigneras}. In
particular, $D\cong \M_2(F)$ if and only if $R=\emptyset$. The
\textit{discriminant} of $D$ is
$$
\fr:=\prod_{\substack{x\in R \\ x\neq \infty}}\wp_x\in A.
$$

Let $D^\times$ be the algebraic group over $F$ defined by
$D^\times(B)=(D\otimes_F B)^\times$ for any $F$-algebra $B$; this is
the multiplicative group of $D$.

\begin{notn} For $a, b\in F^\times$, let $H(a,b)$ be the $F$-algebra with basis $1,i,j,
ij$ (as an $F$-vector space), where $i,j\in H(a,b)$ satisfy:
\begin{itemize}
\item If $q$ is odd,
$$
i^2=a,\quad j^2=b,\quad ij=-ji.
$$
\item If $q$ is even,
$$
i^2+i=a,\quad j^2=b,\quad ij=j(i+1).
$$
\end{itemize}
It is easy to show that $H(a,b)$ is a quaternion algebra; cf.
\cite[pp. 1-5]{Vigneras}.
\end{notn}

\begin{prop}\label{prop-emb}
Let $L$ be a field of degree $n$ over $F$. Then $L$ embeds into $D$,
i.e., there is an $F$-isomorphism of $L$ onto an $F$-subalgebra of
$D$, if and only if $n$ divides $2$ and no place in $R$ splits in
$L$. Any two such $F$-isomorphisms are conjugate in $D$. If $L$ is a
quadratic extension of $F$, then $L$ splits $D$ if and only if $L$
embeds into $D$.
\end{prop}
\begin{proof}
See \cite[(32.15)]{Reiner} and \cite[Thm. I.2.8]{Vigneras}.
\end{proof}

We denote by $\alpha\mapsto \alpha'$ the canonical involution of $D$
\cite[p. 1]{Vigneras}; thus $\alpha''=\alpha$ and
$(\alpha\beta)'=\beta'\alpha'$. The \textit{reduced trace} of
$\alpha$ is $\Tr(\alpha)=\alpha+\alpha'$; the \textit{reduced norm}
of $\alpha$ is $\Nr(\alpha)=\alpha\alpha'$; the \textit{reduced
characteristic polynomial} of $\alpha$ is
$$
f(x)=(x-\alpha)(x-\alpha')=x^2-\Tr(\alpha)x+\Nr(\alpha).
$$
If $D$ is a division algebra and $\alpha\not\in F$, then the field
$F(\alpha)$ generated by $\alpha$ over $F$ is quadratic over $F$,
and the reduced trace and norm of $\alpha$ are simply the images of
$\alpha$ under the trace and norm of $F(\alpha)/F$.


\subsection{$\cD$-elliptic sheaves}\label{ssDES} From now on we assume that $D$ is a division
algebra and $D$ is split at $\infty$ (this is the analogue of
indefinite quaternion algebra over $\Q$). Fix a locally free sheaf
$\cD$ of $\cO_C$-algebras with stalk at the generic point equal to
$D$ and such that $\cD_v:=\cD\otimes_{\cO_C}\cO_v$ is a maximal
order in $D_v:=D\otimes_F F_v$. For $S\subset |C|$, denote
$\cD^S:=\prod_{v\in |C|-S}\cD_v$.

Let $W$ be an $\F_q$-scheme. Denote by $\Frob_W$ its Frobenius
endomorphism, which is the identity on the points and the $q$-th
power map on the functions. Denote by $C\times W$ the fibred product
$C\times_{\Spec(\F_q)}W$.

Let $z:W\to C$ be a morphism of $\F_q$-schemes. A
\textit{$\cD$-elliptic sheaf over $W$}, with pole $\infty$ and zero
$z$, is a sequence $\E=(\cE_i,j_i,t_i)_{i\in \Z}$, where each
$\cE_i$ is a locally free sheaf of $\cO_{C\times W}$-modules of rank
$4$ equipped with a right action of $\cD$ compatible with the
$\cO_C$-action, and where
\begin{align*}
j_i &:\cE_i\to \cE_{i+1}\\
t_i &:\twist{\cE}_{i}:=(\Id_C\times \Frob_{W})^\ast \cE_i\to
\cE_{i+1}
\end{align*}
are injective $\cO_{C\times W}$-linear homomorphisms compatible with
the $\cD$-action. The maps $j_i$ and $t_i$ are sheaf modifications
at $\infty$ and $z$, respectively, which satisfy certain conditions,
and it is assumed that for each closed point $w$ of $W$, the
Euler-Poincar\'e characteristic $\chi(\cE_0|_{C\times w})$ is in the
interval $[0,d)$; we refer to \cite[$\S$2]{LRS} and
\cite[$\S$1]{Hausberger} for the precise definition. Moreover, to
obtain moduli schemes with good properties, at the closed points $w$
of $W$ such that $z(w)\in R$ one imposes an extra condition on $\E$
to be ``special'', which is essentially a requirement that the trace
of the induced action of $g\in \cD_o$, $o\in R$, on the Lie algebra
of $\E$ at $o$ is equal to the reduced trace of $g$ as an element of
the quaternion algebra $D_o$ (see \cite[p. 1305]{Hausberger}). Note
that, unlike the original definition in \cite{LRS}, $\infty$ is
allowed to be in the image of $W$; here we refer to
\cite[$\S$4.4]{BS} for the details. Denote by $\Ell^\cD(W)$ the set
of isomorphism classes of $\cD$-elliptic sheaves over $W$.

\begin{thm}\label{thmMC}
The functor $W\mapsto \Ell^{\cD}(W)$ has a coarse moduli scheme
$X^\cD$, which is projective of pure relative dimension $1$ over $C$
and is smooth over $C-R-\infty$.
\end{thm}
\begin{proof} Let $I$ be a closed non-empty subscheme of $C$ such
that $I\cap (R\cup \infty)=\emptyset$. It is possible to add
level-$I$ structure to the moduli problem, assuming one considers
only those schemes $W$ for which $z(W)\cap I=\emptyset$. The
resulting moduli problem is representable by a projective
$1$-dimensional scheme over $C-I$, which is smooth over
$C-I-R-\infty$; see \cite[Thm. 6.4]{Hausberger} and
\cite[$\S$4.4]{BS}. (The proof of this result is bases on the
techniques in \cite{LRS}.) Dividing by the level structure, we
obtain $X^\cD$ over $C-I-R-\infty$. To extend this scheme over $I$,
repeat the same argument using some $J$ disjoint from $I$, and glue
the resulting scheme to $X^\cD$ over $C-J-I-R-\infty$.
\end{proof}


\subsection{Graphs}\label{SecG} We recall some of the terminology related to graphs, as presented in
\cite{SerreT} and \cite{Kurihara}. A graph $\cG$ consists of a set
of \textit{vertices} $X=\Ver(\cG)$ and a set of \textit{edges}
$Y=\Ed(\cG)$. Every edge $y$ has an \textit{inverse} edge
$\bar{y}\in Y$, an \textit{origin} $o(y)\in X$ and a
\textit{terminus} $t(y)\in X$. Moreover, $\bar{\bar{y}}=y$,
$\bar{y}\neq y$ and $o(y)=t(\bar{y})$. The vertices $o(y)$ and
$t(y)$ are the \textit{extremities} of $y$. Note that it is allowed
for distinct edges $y\neq z$ to have $o(y)=o(z)$ and $t(y)=t(z)$,
and it is also allowed to have $y\in Y$ with $o(y)=t(y)$, in which
case $y$ is called a \textit{loop}. We say that two vertices are
\textit{adjacent} if they are the extremities of some edge. We will
assume that for any $v\in X$ the number of edges $y\in Y$ with
$o(y)=v$ is finite; this number is the \textit{degree} of $v$. A
vertex $v\in X$ is called \textit{terminal} if it has degree $1$. A
graph can be represented by a diagram where a marked point
corresponds to a vertex and a line joining two marker points
corresponds to a set of edges of the form $\{y, \bar{y}\}$; see
\cite[$\S$I.2.1]{SerreT}. A \textit{path} (without backtracking) of
length $m$ between two vertices $v, w$ in $\cG$ is a collection of
edges $y_1,y_2,\dots, y_m\in \Ed(\cG)$ such that $y_i\neq
\bar{y}_{i+1}$, $t(y_i)=o(y_{i+1})$ for $1\leq i\leq m-1$, and
$v=o(y_1)$, $w=t(y_m)$. We assume that $\cG$ is connected, i.e., a
path between any two distinct vertices exists. A path of minimal
length is called a \textit{geodesic}; the distance $d(v, w)$ between
$v$ and $w$ is the length of a geodesic. A graph which has no
non-trivial paths from a vertex to itself is called a \textit{tree}.
Note that in a tree a path between any two vertices is unique. A
graph is \textit{finite} if it has finitely many vertices and edges.
A finite graph $\cG$ can be interpreted as a $1$-dimensional
CW-complex \cite[p. 22]{SerreT}. The \textit{first Betti number}
$h_1(\cG)$ of $\cG$ is the dimension of the homology group $\dim_\Q
H_1(\cG, \Q)$.

$\cG$ is a \textit{graph with lengths} if we are given a map
$$
\ell=\ell_\cG: \Ed(\cG)\to \N=\{1,2,3, \cdots\}
$$
such that $\ell(y)=\ell(\bar{y})$.  An \textit{automorphism} of
$\cG$ is a pair $\phi=(\phi_1, \phi_2)$ of bijections $\phi_1: X\to
X$ and $\phi_2: Y\to Y$ such that $\phi_1(o(y))=o(\phi_2(y))$,
$\overline{\phi_2(y)}=\phi_2(\bar{y})$, and
$\ell(y)=\ell(\phi_2(y))$.

Let $\G$ be a group acting on a graph $\cG$ (i.e., $\G$ acts via
automorphisms). We say that $v, w\in X$ are $\G$-\textit{equivalent}
if there is $\gamma \in \G$ such that $\gamma v=w$; similarly, $y,
z\in Y$ are $\G$-\textit{equivalent} if there is $\gamma\in \G$ such
that $\gamma y=z$. For $v\in X$, denote
$$
\G_v=\Stab_\G(v)=\{\gamma\in \G\ |\ \gamma v=v\}
$$
the stabilizer of $v$ in $\G$. Similarly, let $\G_{y}=\G_{\bar{y}}$
be the stabilizer of $y\in Y$. $\G$ acts with \textit{inversion} if
there is $\gamma\in \G$ and $y\in Y$ such that $\gamma y=\bar{y}$.
If $\G$ acts without inversion, then we have a natural quotient
graph $\G\bs\cG$ such that $\Ver(\G\bs \cG)=\G\bs \Ver(\cG)$ and
$\Ed(\G\bs \cG)=\G\bs \Ed(\cG)$.

\begin{defn}(cf. \cite[p. 70]{SerreT}) Let $\cO$ be a complete discrete valuation
ring with fraction field $K$, finite residue field $k\cong \F_q$ and
a uniformizer $\pi$. Let $V$ be a two-dimensional vector space over
$K$. A \textit{lattice} of $V$ is a free rank-$2$ $\cO$-submodule of
$V$ which generates the $K$-vector space $V$. Two lattices $\La$ and
$\La'$ are \textit{homothetic} if there is $x\in K^\times$ such that
$\La'=x\La$. We denote the homothety class of $\La$ by $[\La]$.

Let $\cT$ be the graph whose vertices $\Ver(\cT)=\{[\La]\}$ are the
homothety classes of lattices in $V$, and two vertices $[\La]$ and
$[\La']$ are adjacent if we can choose representatives $L\in [\La]$
and $L'\in [\La']$ such that $L'\subset L$ and $L/L'\cong k$. One
shows that $\cT$ is an infinite tree in which every vertex has
degree $(q+1)$. This is the \textit{Bruhat-Tits tree} of
$\PGL_2(K)$. The group $\GL_2(K)$, as the group of linear
automorphisms of $V$, naturally acts on $\cT$ and preserves the
distances between vertices.
\end{defn}

\begin{lem}\label{lem1.6}
Let $g\in \GL_2(K)$ and $v\in \Ver(\cT)$. Then
$$
d(v, gv)\equiv \ord_K(\det(g))\ (\mod\ 2).
$$
\end{lem}
\begin{proof}
See Corollary on p. 75 in \cite{SerreT}.
\end{proof}


\subsection{Admissible curves}\label{ssAC} Let $\cO$ be a complete discrete valuation
ring with fraction field $K$, finite residue field $k$ and a
uniformizer $\pi$. Let $S=\Spec(\cO)$, $\cO^\un$ be the maximal
unramified extension of $\cO$, $\widehat{\cO}^\un$ be its
completion, and $\bar{k}=\cO^\un/\pi\cO^\un$.

\begin{defn}\label{def-AC} (cf. \cite[$\S$3]{JL}) A curve $X\to S$ is called \textit{admissible} if
\begin{enumerate}
\item $X$ is proper and flat over $S$ and its generic fibre $X_K$ is a
smooth curve.
\item The special fibre $X_k$ is reduced with normal crossing
singularities, and every irreducible component is isomorphic to
$\P^1_k$.
\item If $x$ is a double point on the special fibre of $X$, then there
exists a unique integer $m_x$ for which the completion of
$\cO_{x,X}\otimes_\cO \widehat{\cO}^\un$ is isomorphic to the
completion of $\widehat{\cO}^\un[t,s]/(ts-\pi^{m_x})$.
\end{enumerate}
\end{defn}

The \textit{dual graph} $\cG=\Gr(X)$ of $X$ is the following graph
with lengths. The vertices of $\cG$ are the irreducible components
of $X_k$. The edges of $\cG$, ignoring the orientation, are the
singular points of $X_k$. If $x$ is a double point and $\{y,
\bar{y}\}$ is the corresponding edge of $\cG$, then the extremities
of $y$ and $\bar{y}$ are the irreducible components passing through
$x$; choosing between $y$ or $\bar{y}$ corresponds to choosing one
of the branches through $x$. Finally, $\ell(y)=\ell(\bar{y})=m_x$.

The admissible curves are closely related to Mumford curves. Let
$\G\subset \PGL_2(K)$ be a discrete subgroup with compact quotient.
There is an admissible curve $X_\G$ over $S$ uniquely determined by
$\G$. $X_\G$ is obtained as follows (see \cite[$\S$3]{Kurihara}).
Let $\G_1$ be a torsion-free normal subgroup of $\G$ with finite
index (such a subgroup always exists). Mumford constructed a formal
scheme $\widehat{\Omega}$ over $\Spf(\cO)$ such that the dual graph
of its closed fibre is isomorphic to the Bruhat-Tits tree $\cT$ of
$\PGL_2(K)$, and proved that there exists a unique admissible curve
$X_{\G_1}$ over $S$ whose completion along its special fibre is
isomorphic to the quotient $\G_1\bs \widehat{\Omega}$ in the
category of formal schemes. Now define $X_\G$ as the quotient
$X_{\G_1}/(\G/\G_1)$; this is independent of the choice of $\G_1$.
Next, $\G$ naturally acts on $\cT$. Assume $\G$ acts without
inversion. We assign lengths to the edges of the quotient graph
$\G\bs \cT$: for $y\in \Ed(\G\bs \cT)$ let
$\ell(y)=\#\G_{\tilde{y}}$, where $\tilde{y}$ is a preimage of $y$
in $\cT$.

\begin{thm}\label{thm-Kurihara} $\Gr(X_\G)\cong \G\bs
\cT$ as graphs with lengths. This implies that the genus of
$(X_\G)_K$ is equal to $h_1(\G\bs \cT)$.
\end{thm}
\begin{proof}
This is \cite[Prop. 3.2]{Kurihara} (note that since $\G$ acts
without inversion on $\cT$ the graph $(\G\bs\cT)^\ast$ in
\cite{Kurihara} is $(\G\bs\cT)$ itself).
\end{proof}

\begin{rem}
$\widehat{\Omega}$ is the formal scheme associated to Drinfeld's
non-archimedean half-plane $\Omega=\P^{1,\an}_K-\P^{1,\an}_K(K)$
over $K$. For the description of the rigid-analytic structure of
$\Omega$ and the construction of $\widehat{\Omega}$ we refer to
\cite[Ch. I]{BC}.
\end{rem}


\section{Places of good reduction}

\begin{thm}\label{thm2.1} Let $o\in |C|-R-\infty$. Let $K$ be a finite extension of
$F_o$ of relative degree $f=f(K/F_o)$ and ramification index
$e=e(K/F_o)$.
\begin{itemize}
\item If $f$ is even, then $X^\cD(K)\neq \emptyset$.
\item If $f$ is odd, then $X^\cD(K)=\emptyset$ if and only if for
every $\alpha$ satisfying a polynomial of the form
$$
X^2+aX+c\wp_o^f \quad \text{with }a\in A \text{ and } c\in
\F_q^\times,
$$
and such that $F(\alpha)$ is quadratic over $F$, either some place
in $(R\cup \infty)$ splits in $F(\alpha)$, or $\wp_o$ divides
$\alpha$ and $o$ splits in $F(\alpha)$.
\end{itemize}
\end{thm}

Before proving the theorem we make some comments about its
consequences.

\begin{rem}
Theorem \ref{thm2.1} implies that if $q$ is even then
$X^\cD(F_o)\neq \emptyset$, and hence also $X^\cD(K)\neq \emptyset$
for any extension $K$ of $F_o$. Indeed, if $q$ is even, then
$X^2+\wp_o=0$ is purely inseparable over $F$, so all the places in
$(R\cup \infty\cup o)$ ramify in $F(\sqrt{\wp_o})$.
\end{rem}

\begin{rem}
If $q$ is odd, then to determine whether $X^\cD(K)=\emptyset$ one
needs to examine only a finite number of quadratic extensions of
$F$. Indeed, $\infty$ splits in $F(\alpha)$, with $\alpha$
satisfying $X^2+aX+c\wp_o^f=0$, once $\deg(a)>f\cdot \deg(o)/2$.
\end{rem}

\begin{rem}
Assume $q$ is odd and fix a non-square $c\in \F_q^\times$. Using the
Chinese Remainder Theorem, it is easy to show that if
$\deg(o)>2\deg(\fr)$ then there exists $a\in A$ such that
$X^2+aX+c\wp_o$ is irreducible over $F_v$ for $v\in R\cup \infty$.
Theorem \ref{thm2.1} then implies that $X^\cD(F_o)\neq \emptyset$.
The same conclusion can be drawn also using the Weil bound. The
genus $g$ of $X^\cD$ is approximately $q^{\deg{\fr}}$ (see Corollary
\ref{cor5.4}), so $q_o+1-2g\sqrt{q_o}>0$ if $\deg(o)>2\deg(\fr)$.
Then the Weil bound implies $X^\cD(\F_o)\neq \emptyset$, so
$X^\cD(F_o)\neq \emptyset$ by Hensel's lemma (cf. the proof of
Theorem \ref{thm2.1}).
\end{rem}

\begin{rem}
If $q$ is odd, then for a fixed $K$ with odd relative degree there
always exist (infinitely many) $X^\cD$ with good reduction at $o$
such that $X^\cD(K)=\emptyset$. Indeed, let $S(a,c)\subset |C|$ be
the set of primes splitting in the quadratic extension generated by
the solutions of $X^2+aX+c\wp_o^f=0$. Each set $S(a,c)$ is infinite,
but there are only finitely many $a,c$ with $c\in \F_q^\times$ and
$\deg(a)\leq f\cdot \deg(o)/2$. Hence we can choose $R\subset
|C|-\infty-o$ of even cardinality which has non-empty intersection
with each $S(a,c)$ (and obviously there are infinitely many such
$R$). By Theorem \ref{thm2.1}, the corresponding $X^\cD$ will have
no $K$-rational points.
\end{rem}

\begin{example}
Let $q=3$ and $R=\{x,y\}$ with $\wp_x=T$ and $\wp_y=T-1$. Let
$\wp_o=T-2$. In the extension generated by the solutions of
$X^2+X+(T-2)=0$ the places $x$ and $\infty$ ramify, and $y$ remains
inert. On the other hand, $o$ splits but $\wp_o$ does not divide
$\alpha$. Hence $X^\cD(F_o)\neq \emptyset$.
\end{example}

\begin{proof}[Proof of Theorem \ref{thm2.1}] We start with a reduction. Denote
$M:=X^\cD\times_C\Spec(\F_o)$. By Theorem \ref{thmMC}, $X^\cD$ has
good reduction at $o$, so by the geometric version of Hensel's lemma
\cite[Lem. 1.1]{JL}, $X^\cD(K)\neq \emptyset$ if and only if
$M(\F_o^{(f)})\neq \emptyset$.

Let $k$ be a fixed algebraic closure of $\F_o$. In \cite[Ch.
9]{LRS}, the authors, following Drinfeld, develop a Honda-Tate type
theory for $\cD$-elliptic sheaves of characteristic $o$ over $k$. A
foundational block of this theory is a construction which attached
to each $\cD$-elliptic sheaf over $k$ a pair $(\tF, \tP)$, called
\textit{$(D,\infty,o)$-type}, having the following properties, cf.
\cite[(9.11)]{LRS}:
\begin{enumerate}
\item $\tF$ is a separable field extension of $F$ of degree dividing
$2$;
\item $\tP\in \tF^\times\otimes_\Z \Q$ is such that $\tP\not\in F^\times\otimes_\Z
\Q$ unless $\tF=F$;
\item $\infty$ does not split in $\tF$;
\item The valuations of $F$ naturally extend to the group $\tF^\times\otimes_\Z\Q$.
There are exactly two places $\tilde{x}=\tilde{\infty}$ and
$\tilde{x}=\tilde{o}$ of $\tF$ such that $\ord_{\tilde{x}}(\Pi)\neq
0$. The place $\tilde{\infty}$ is the unique place of $\tF$ over
$\infty$ and $\tilde{o}$ divides $o$. Moreover,
$$
\ord_{\tilde{\infty}}(\tP)\cdot \deg(\tilde{\infty})=-[\tF:F]/2.
$$
\item For each place $x$ of $F$ and each place $\tilde{x}$ of $\tF$
dividing $x$, we have
$$
(2[\tF_{\tilde{x}}:F_x]/[\tF:F])\cdot \inv_x(D)\in \Z.
$$
\end{enumerate}

Next, using (1)-(5), one deduces the following property of $\tF$
depending on the value of $h:=2[\tF_{\tilde{o}}:F_o]/[\tF:F]$ (see
\cite[Lem. 4.1]{PapGenus}). If $h=1$, then $\tF$ is a separable
quadratic extension of $F$ in which $o$ splits and each $x\in R\cup
\infty$ does not split. If $h=2$, then $\tF=F$ and the
$(D,\infty,o)$-type $(F, \Pi)$ is unique.

By \cite[Thm 9.13]{LRS}, the map which associates to a
$\cD$-elliptic sheaf of characteristic $o$ over $k$ its
$(D,\infty,o)$-type decomposes the set $M(k)$ into a disjoint union
of non-empty subsets (called, \textit{isogeny classes})
$$
M(k)=\bigsqcup M(k)_{(\tF,\tP)}.
$$
The union is over all $(D,\infty,o)$-types, the points in the subset
$M(k)_{(\tF,\tP)}$ have $(D,\infty,o)$-type $(\tF, \tP)$, and each
such subset is non-empty. Moreover, the geometric Frobenius
$\Frob_o: x\mapsto x^{-q_o}$ preserves $M(k)_{(\tF,\tP)}$.

We will examine when a given isogeny class contains an
$\F_o^{(f)}$-rational points. For this we will employ the
group-theoretic description of each isogeny class along with the
action of the geometric Frobenius given in \cite[$\S$10]{LRS}.

We start with the case $F=\tF$. Let $\bD$ be the quaternion algebra
over $F$ which is ramified exactly at the places $R\cup o\cup
\infty$. There is a bijection
$$
M(k)_{(F,\Pi)}\overset{\sim}{\To}\bD^\times(F)\bs\left[\bD^\times(\A^{\infty,o})/(\cD^{\infty,o})^\times\times
\Z\right],
$$
where $\bD^\times(F)$ is embedded diagonally into
$\bD^\times(\A^{\infty,o})$, and $(\cD^{\infty,o})^\times$ is
considered as a subgroup of $\bD(\A^{\infty, o})$ via an isomorphism
$\bD(\A^{\infty, o})\cong D(\A^{\infty, o})$. $\bD^\times(F)$ act on
$\bD^\times(\A^{\infty,o})/(\cD^{\infty,o})^\times$ by left
multiplication and on $\Z$ via the composition $\ord_o\circ \Nr$.
The action of $\Frob_o$ on $M(k)_{(F,\Pi)}$ corresponds to the
translation by $1$ on $\Z$.

Hence $M(\F_o^{(f)})_{(F,\Pi)}\neq \emptyset$ if and only if there
is an element $\delta\in \bD^\times(F)\cap (\cD^{\infty,o})^\times$
whose reduced norm generates the ideal $\fp_o^f$, cf. \cite[p.
278]{LRS}. This is equivalent to the existence of $\delta\in
\bD^\times(F)$ such that $\Tr(\delta)=a\in A$, $\Nr(\delta)=b\in A$
and $(b)=\fp_o^f$.

If $f$ is even, such $\delta$ always exists, since $\wp_o^{f/2}\in
F^\times$ satisfies the required conditions.

Now suppose $f$ is odd, and $\delta$ with the required properties
exists. If $\delta\in F$, then $\Nr(\delta)$ is a square in $F$
which contradicts $f$ being odd. Hence the reduced characteristic
polynomial of $\delta$ is of the form $P_\delta(X)=X^2+aX+c\wp_o^f$,
with $a\in A$ and $c\in \F_q^\times$, and the solutions of this
polynomial generate a quadratic subfield $F':=F(\delta)$ of $\bD$.
In particular, $F'$ is a splitting field for $\bD$, so by
Proposition \ref{prop-emb} the places in $(R\cup \infty \cup o)$ do
not split in $F'$. Conversely, suppose there is a polynomial
$P(X)=X^2+aX+b$ whose solutions generate a quadratic extension $F'$
of $F$ in which the places in $(R\cup \infty \cup o)$ do not split
and $(b)=\fp_o^f$. Again, by Proposition \ref{prop-emb}, $F'$ embeds
into $\bD$. If $\delta$ is such that $P(\delta)=0$, then the image
of $\delta$ in $\bD$ is an element which is integral over $A$ and
has reduced norm generating $\fp_o^f$.

We conclude that if $f$ is even, then $M(\F_o^{(f)})_{(F,\Pi)}\neq
\emptyset$. On the other hand, if $f$ is odd, then
$M(\F_o^{(f)})_{(F,\Pi)}= \emptyset$ if and only if for every
$\alpha$ which is quadratic over $F$ and satisfies a polynomial of
the form $X^2+aX+c\wp_o^f$, with $a\in A$ and $c\in \F_q^\times$,
some place in $(R\cup \infty \cup o)$ splits in $F(\alpha)$.

\vspace{0.1in}

Next, consider a $(D,\infty,o)$-type $(\tF, \tP)$ such that $\tF$ is
quadratic over $F$. The valuations of $\tP$ at the places of $\tF$
are zero except at the unique place $\widetilde{\infty}$ over
$\infty$ and a place $\tilde{o}$ dividing $o$. Since $o$ splits in
$\tF$, let $\tilde{\tilde{o}}\neq \tilde{o}$ be the other place of
$\tF$ over $o$. Identify $\tF_{\tilde{\tilde{o}}}$ with $F_o$. There
is a bijection
$$
M(k)_{(\tF,\tP)}\overset{\sim}{\To}\tF^\times\bs\left[D^\times(\A^{\infty,o})/(\cD^{\infty,o})^\times\times
(F_o^\times/\cO_o^\times)\times \Z\right],
$$
where $\tF^\times$ is embedded diagonally into
$D^\times(\A^{\infty,o})$ (such embedding exists since $\tF$ is a
splitting field for $D$) and into $F_o^\times$ via
$\tF\hookrightarrow \tF_{\tilde{\tilde{o}}}\cong F_o$. $\tF^\times$
acts on $D^\times(\A^{\infty,o})/(\cD^{\infty,o})^\times$ and
$(F_o^\times/\cO_o^\times)$ via left multiplication, and on $\Z$ via
the composition $\ord_o\circ \Nr_{\tF/F}$. The action of $\Frob_o$
on $M(k)_{(\tF,\tP)}$ corresponds to the translation by $1$ on $\Z$.

The previous paragraph implies that $M(\F_o^{(f)})_{(\tF,\tP)}\neq
\emptyset$ if and only if there is $\delta\in \tF$ which has zero
valuations away from $\tilde{o}$ and $\tilde{\infty}$ and
$\Nr_{\tF/F}(\delta)$ generates $\fp_o^f$. The characteristic
polynomial of such an element has the form
$P_\delta(X)=X^2+aX+c\wp_o^f$, where $a\in A$ and $c\in
\F_q^\times$. If $\alpha$ is a root of this polynomial, then the
places in $(R\cup \infty)$ do not split in $F(\alpha)$ and $o$
splits (since $F(\alpha)\cong \tF$). Moreover, $\wp_o$ does not
divide $\alpha$, as otherwise $\delta$ will have non-zero valuations
at both places over $o$. Conversely, suppose there is a polynomial
$P(X)=X^2+aX+b$ whose solutions generate a quadratic extension $F'$
of $F$ in which the places in $(R\cup \infty)$ do not split, $o$
splits, and $(b)=\fp_o^f$. Let $\delta$ be a root of $P(X)$. Let
$\tilde{o}$ and $\tilde{\tilde{o}}$ be the two places in $F'$ over
$F$. The valuations of $\delta$ at the places of $F'$ are zero,
except at the place $\tilde{\infty}$ over $\infty$, and one or both
of $\tilde{o}\cup \tilde{\tilde{o}}$. Suppose
$\ord_{\tilde{\tilde{o}}}(\delta)=0$. Then, since there is a
$(D,\infty,o)$-type $(\tF, \tP)$ with $\tF=F'$ and
$\ord_{\tilde{o}}(\tP)\neq 0$, $M(\F_o^{(f)})_{(\tF,\tP)}$ will be
non-empty.

We conclude that $\bigsqcup_{\tF\neq F}
M(\F_o^{(f)})_{(\tF,\tP)}=\emptyset$ if and only if for every
$\alpha$ which is quadratic over $F$ and satisfies a polynomial of
the form $X^2+aX+c\wp_o^f$, with $a\in A$ and $c\in \F_q^\times$,
one of the following holds: (1) some place in $(R\cup \infty)$
splits in $F(\alpha)$; (2) $o$ does not split in $F(\alpha)$; (3)
$\wp_o$ divides $\alpha$. Combining this with the conditions from
the case $\tF=F$, we obtain Theorem \ref{thm2.1}.
\end{proof}

\begin{rem}
The analogue of Theorem \ref{thm2.1} over $\Q$ is \cite[Thm.
2.5]{JL}. Jordan and Livn\'e prove this theorem by
``deconstructing'' the zeta-function of the reduction of Shimura
curve at a place of good reduction. The approach via Honda-Tate
theory, as is done presently, has the advantage of giving an
interpretation to the conditions appearing in Theorem \ref{thm2.1}
in terms of isogeny classes.
\end{rem}


\section{Finite places of bad reduction}

\begin{thm}\label{thm3.1} Let $o\in R$. Let $K$ be a finite extension of
$F_o$ of relative degree $f=f(K/F_o)$ and ramification index
$e=e(K/F_o)$.
\begin{enumerate}
\item If $f$ is even, then $X^\cD(K)\neq \emptyset$.
\item If $f$ is odd and $e$ is even, then $X^\cD(K)=\emptyset$ if and only if
in every quadratic extension $F(\sqrt{c\wp_o})/F$, with $c\in
\F_q^\times$, some place in $(R-o)\cup \infty$ splits.
\item If $f$ and $e$ are both odd, then $X^\cD(K)=\emptyset$.
\end{enumerate}
\end{thm}

\begin{rem}
Note that if $q$ is even then $F(\sqrt{c\wp_o})=F(\sqrt{\wp_o})$ is
a purely inseparable extension, so the places $(R-o)\cup \infty$
ramify; in particular, if $f$ or $e$ are even, then $X^\cD(K)\neq
\emptyset$. If $q$ is odd, then there are only two extensions
$F(\sqrt{c\wp_o})$, namely $F(\sqrt{\wp_o})$ and
$F(\sqrt{\xi\wp_o})$, where $\xi\in \F_q^\times$ is a fixed
non-square.
\end{rem}

\begin{rem}
An immediate consequence of Part (3) of Theorem \ref{thm3.1} is
$X^\cD(F_v)=\emptyset$ for $v\in R$. In particular,
$X^\cD(F)=\emptyset$.
\end{rem}

The rest of this section is devoted to the proof of Theorem
\ref{thm3.1}. We start by recalling the analogue of
\u{C}erednik-Drinfeld uniformization for $X^\cD\otimes F_o$ due to
Hausberger \cite{Hausberger}.

Let $\bD$ be the quaternion algebra over $F$ which is ramified
exactly at $(R-o)\cup \infty$. Fix a maximal $A$-order $\fD$ in
$\bD(F)$, and denote
\begin{align*}
&A^o=A[\wp_o^{-1}];\\
&\fD^o=\fD\otimes_A A^o;\\
&\G=(\fD^o)^\times;\\
&\G_+=\{\gamma\in \G\ |\ \ord_o(\Nr(\gamma))\in 2\Z\}.
\end{align*}

If we fix an identification of $\bD_o$ with $\M_2(F_o)$, then $\G$
and $\G_+$ are discrete cocompact subgroup of $\GL_2(F_o)$
containing $\begin{pmatrix} \wp_o & 0\\ 0 & \wp_o\end{pmatrix}$.
Hence both $\G$ and $\G_+$ act naturally on Drinfeld's $o$-adic
half-plane $\Omega:=\P^{1,\an}_{F_o}-\P^{1,\an}_{F_o}(F_o)$ and its
formal analogue $\widehat{\Omega}$ (see $\S$\ref{ssAC}). Denote by
$\Fr_o:\widehat{\cO}_o^\un\to \widehat{\cO}_o^\un$ the lift of the
Frobenius homomorphism $x\mapsto x^{q_o}$ on $\overline{\F}_o$ to an
$\cO_o$-homomorphism. Consider the formal scheme
$$
\widehat{\Omega}^\un:=\widehat{\Omega}\times_{\Spf(\cO_o)}\Spf(\widehat{\cO}_o^\un)
$$
over $\Spf(\cO_o)$. Define an action of $\GL_2(F_o)$ on
$\widehat{\Omega}^\un$ as follows: For an element $\gamma\in
\GL_2(F_o) $ denote by $[\gamma]$ the element defined by $\gamma$ in
$\PGL_2(F_o)$ and set $n(\gamma)=-\ord_o(\det(\gamma))$. Then define
$$
\gamma:(x,u)\mapsto ([\gamma]x, \Fr_o^{n(\gamma)}u)
$$
for $x\in \widehat{\Omega}$ and $u\in \Spf(\widehat{\cO}_o^\un)$,
where $[\gamma]$ acts on $x$ via the natural action. Note that for
$\gamma\in \G$, $\Nr(\gamma)=\det(\iota(\gamma))$, where
$\iota:\G\hookrightarrow \GL_2(F_o)$.

Let $M:=X^\cD\times_C \Spec(\cO_o)$, and denote by $\widehat{M}$ the
completion of $M$ along its special fibre. Note that
$\bD^\times(\A^{\infty,o})\cong D^\times(\A^{\infty,o})$, so the
group $(\cD^{\infty,o})^\times$ can be considered as a subgroup of
$\bD^\times(\A^{\infty,o})$. Let
$$
\cZ:=\bD^\times(F)\bs \bD^\times(\A^{\infty})/
(\cD^{\infty,o})^\times.
$$
The set $\cZ$ is equipped with an obvious right action of
$\bD^\times(F_o)=\GL_2(F_o)$. With this notation, we have the
following uniformization due to Hausberger \cite[Thm.
8.1]{Hausberger}:
\begin{equation}\label{eq-Haus}
\widehat{M}\cong \GL_2(F_o)\bs[\widehat{\Omega}^\un\times \cZ].
\end{equation}
Since by the Strong Approximation Theorem \cite[p. 81]{Vigneras}
$$
\bD^\times(\A^{\infty})=D^\times(F)\cdot \GL_2(F_o)\cdot
(\cD^{\infty,o})^\times,
$$
the isomorphism (\ref{eq-Haus}) can be rewritten as
$$
\widehat{M}\cong \G\bs \widehat{\Omega}^\un.
$$
Denote by $\cO_o^{(2)}$ the quadratic unramified extension of
$\cO_o$. Let $W=\G/\G_+=\{1, w_o\}$, where $w_o$ may be represented
by any element $\gamma_o\in \G$ whose norm generates $\fp_o$. Now
the argument in \cite[p. 143]{BC} implies
\begin{equation}\label{eq-3}
\widehat{M}\cong W\bs [(\G_+\bs\widehat{\Omega})\otimes
\cO_o^{(2)}],
\end{equation}
where the action of $W$ is given by $w_o:(x, u)\mapsto (w_o x,
\Fr_o^{-1} u)$ for a point $x\in \widehat{\Omega}$ and $u\in
\Spec(\cO_o^{(2)})$.

The existence of the uniformization (\ref{eq-3}) implies that $M$ is
an admissible curve, cf. $\S$\ref{ssAC}. The vertices of the
Bruhat-Tits tree $\cT$ of $\PGL_2(F_o)$ can be partitioned into two
disjoint classes such that the vertices in the same class are an
even distance apart; see \cite[p. 71]{SerreT}. Therefore, by Lemma
\ref{lem1.6}, $\G_+$ preserve the two classes of $\Ver(\cT)$ and
$w_o$ interchanges these two classes. This implies that $\G_+$ acts
without inversion on $\cT$, and that $w_o$ has no fixed vertices in
$\G_+\bs \cT$.

\begin{prop}\label{prop3.2} With notation of Theorem \ref{thm3.1},
if $f$ is even, then $M(K)\neq \emptyset$. If $f$ is odd, then
$M(K)\neq \emptyset$ if and only if there is an edge $y\in
\Ed(\G_+\bs\cT)$ such that $e\cdot \ell(y)$ is even and
$w_o(y)=\bar{y}$.
\end{prop}
\begin{proof} This follows from the previous discussion combined with the arguments in the proofs
of Theorems 5.1 and 5.2 in \cite{JL} (essentially verbatim).
\end{proof}

Proposition \ref{prop3.2} establishes Part (1) of Theorem
\ref{thm3.1}. Next, we examine when there is an edge $y\in
\Ed(\G_+\bs \cT)$ such that $w_o(y)=\bar{y}$. Let $y\in \Ed(\G_+\bs
\cT)$ and $\tilde{y}$ be a preimage of $y$ in $\cT$. If $w_o
y=\bar{y}$ then there exits $\gamma\in \G$ such that
$\gamma\tilde{y}=\bar{\tilde{y}}$. Let $v, w\in \Ver(\cT)$ be the
extremities of $\tilde{y}$. Then $\gamma(v)=w$, $\gamma(w)=v$, and
so $\gamma^2(v)=v$. This implies $\gamma^2\in F_o^\times
\mu\fD_o^\times\mu^{-1}$ for some $\mu\in \GL_2(F_o)$. From this we
conclude that $\gamma^2=\wp_o^nc$, where $n\in \Z$ and $c\in
\mu\fD_o^\times\mu^{-1}$. Since the distance between $v$ and $\gamma
v$ is $1$, $\ord_o(\Nr(\gamma))=n$ is odd. Put $n=2m+1$ and replace
$\gamma$ by $\wp_o^{-m}\gamma$. Then we get $\gamma^2=c\wp_o$ and
$c\in \mu\fD_o^\times\mu^{-1}\cap \Gamma\subset (\fD')^\times$,
where $\fD'$ is a maximal $A$-order in $\bD$. Since $\bD$ is
ramified at $\infty$, $(\fD')^\times\cong \F_q^\times$ or
$\F_{q^2}^\times$; see \cite[p. 383]{DvG}. Therefore, $c$ is
algebraic over $\F_q$. The field $L:=F(\gamma)$ embeds into $\bD$,
so by Proposition \ref{prop-emb}, $L/F$ is a quadratic extension.
Since $o$ ramifies in $L$, $\F_q$ is algebraically closed in $L$. We
get $c\in \F_{q^2}^\times\cap L=\F_q^\times$. Conversely, suppose
there is $c\in \F_q^\times$ such that $F(\sqrt{c\wp_o})$ embeds into
$\bD$. The image of $\sqrt{c\wp_o}$ is contained in some maximal
$A^o$-order of $\bD$. Since $\bD$ is split at $o$ and $\Pic(A^o)=1$,
a theorem of Eichler implies that all such orders are conjugate in
$\bD$; see \cite[Cor. III.5.7]{Vigneras}. Therefore, we can assume
that $\fD^o$ contains an element $\gamma$ such that $\gamma^2=c
\wp_o$. Such an element is obviously in $\G$. We can choose an
embedding $\G\hookrightarrow \GL_2(F_o)$
such that $\gamma$ maps to the matrix $\begin{pmatrix} 0 & 1\\
c\wp_o & 0\end{pmatrix}$. Using this matrix representation, it is
easy to check that there is an edge $\tilde{y}\in \Ed(\cT)$ such
that $\gamma \tilde{y}=\bar{\tilde{y}}$. Then the image $y$ of
$\tilde{y}$ in $\G_+\bs \cT$ will satisfy $w_o y=\bar{y}$. Finally,
observe that $F(\sqrt{c\wp_o})$ embeds into $\bD$ if and only if the
places in $(R-o)\cup \infty$ do not split in this extension
(Proposition \ref{prop-emb}). Combining the previous discussion with
Proposition \ref{prop3.2}, we establish Part (2) of Theorem
\ref{thm3.1}.

\vspace{0.1in}

From now on we assume that $f$ and $e$ are both odd. We start the
proof of Part (3) of Theorem \ref{thm3.1} with a lemma:

\begin{lem}\label{lem3.4}
An edge of $\G_+\bs\cT$ has length $1$ or $q+1$. The number of edges
of length $q+1$ is equal to
$$
2^{\# R-1}\cdot \Odd(R-o)\cdot (1-\Odd(o)).
$$
\end{lem}
\begin{proof} Let $k$ be a fixed algebraic
closure of $\F_o$. The singular points of $M\times_{\cO_o} k$
represent the isomorphism classes of $\cD$-elliptic sheaves of
characteristic $o$ over $k$ which are \textit{superspecial}; see
\cite[Prop. 5.9]{PapEndom}. Let $x$ be a singular point and let $\E$
be the superspecial $\cD$-elliptic sheaf corresponding to $x$. There
is a natural notion of an automorphism group $\Aut(\E)$ of $\E$.
This is a finite group isomorphic to $\F_q^\times$ or
$\F_{q^2}^\times$; see \cite[Lem. 5.7]{PapEndom}. An argument
similar to \cite[Thm. VI.6.9]{DR} implies that the integer $m_x$
attached to $x$ in Definition \ref{def-AC} is equal to $\#
(\Aut(\E)/\F_q^\times)$. Thus, $m_{x}$ is equal to $1$ or $q+1$.

Next, since $\E$ is superspecial, $\Aut(\E)$ is the group of units
in an Eichler $A$-order $\cH$ in $\bD$ of level $\fp_o$; see
\cite[Thm. 5.3]{PapEndom}. Let $I_1,\cdots, I_h$ represent the
isomorphism classes of locally free left $\cH$-ideals. For $1\leq
i\leq h$, we denote with $\cH_i$ the right order of the respective
$I_i$. Each $\cH_i$ is an Eichler $A$-order of level $\fp_o$. By
\cite[Thm. 5.4]{PapEndom}, the singular points $x_1,\dots, x_h$ of
$M\times_{\cO_o} k$ are in natural bijection with $I_1,\cdots, I_h$
and $m_{x_i}=\#(\cH_i^\times/\F_q^\times)$. Now the class number
formula for Eichler orders \cite[Thm. 9]{DvG} implies
$$
\#\{1\leq i\leq h\ |\ m_{x_i}=q+1\}= 2^{\# R}\cdot \Odd(R-o)\cdot
(1-\Odd(o)).
$$

Finally, Theorem \ref{thm-Kurihara}, \cite[Prop. 3.7]{JL} and
(\ref{eq-3}) imply that the dual graph $\Gr(M)$, as a graph with
lengths, is isomorphic to $\G_+\bs \cT$, so the claim of the lemma
follows from the previous paragraph.
\end{proof}

Thanks to Proposition \ref{prop3.2} and Lemma \ref{lem3.4}, to show
that $M(K)=\emptyset$ when $f$ and $e$ are both odd we can assume
that $q$ is odd and $\deg(o)$ is even.

Let $\xi$ be a fixed non-square in $\F_q^\times$. Let $y\in
\Ed(\G_+\bs\cT)$ be an edge with $\ell(y)=q+1$. Let $\tilde{y}\in
\Ed(\cT)$ be any edge of $\cT$ lying above $y$. By the argument in
the proof of Lemma \ref{lem3.4}, there exists $\gamma\in \G_+$ such
that $\gamma^2=\xi$ and
$\Stab_{\G_+}(\tilde{y})=\F_q(\gamma)^\times$.

Now let $y\in \Ed(\G_+\bs\cT)$ be an edge such that
$w_o(y)=\bar{y}$, and let $\tilde{y}\in \Ed(\cT)$ be any edge of
$\cT$ lying above $y$. From our earlier discussions we know that
there exists $\gamma_0\in \G$ such that $\gamma_0^2=c\wp_o$ with
$c\in \F_q^\times$, and $\gamma_0 \tilde{y}=\bar{\tilde{y}}$. Since
$\bD$ is ramified at $\infty$, in the extension $F(\gamma_0)/F$ the
place $\infty$ does not split. On the other hand, by assumption
$\deg(o)$ is even, so we can assume $c=\xi$ (as $\infty$ splits in
the extension $F(\sqrt{\wp_o})/F$).

The previous discussion, combined with Proposition \ref{prop3.2},
implies that if $M(K)\neq 0$, then there exists $\tilde{y}\in
\Ed(\cT)$ and $\gamma, \gamma_0\in \G$ such that
\begin{equation}\label{eq-4}
\gamma^2=\xi,\quad \gamma_0^2=\xi \wp_o,\quad \gamma
\tilde{y}=\tilde{y},\quad \gamma_0\tilde{y}=\bar{\tilde{y}}.
\end{equation}
Note that $\gamma_0^{-1}\gamma\gamma_0$ is in $\G$ and moreover,
$\ord_o(\Nr(\gamma_0^{-1}\gamma\gamma_0))=\ord_o(\Nr(\gamma))=0$, as
$\gamma$ is algebraic over $\F_q$. Hence
$\gamma_0^{-1}\gamma\gamma_0\in \G_+$. On the other hand,
$\gamma_0^{-1}\gamma\gamma_0\cdot \tilde{y}=\tilde{y}$. Therefore,
$$
\gamma_0^{-1}\gamma\gamma_0\in  (A^o)^\times \F_q(\gamma)^\times.
$$
We conclude that there is $s\in \Z$ and $a,b\in \F_q$ ($a, b$ are
not both zero) such that
$$
\gamma\gamma_0=\wp_o^s\gamma_0(a+b\gamma).
$$
Comparing the norms of both sides, we must have $s=0$. Thus,
\begin{equation}\label{eq-1}
\gamma\gamma_0=\gamma_0(a+b\gamma).
\end{equation}
Note that $\gamma$ and $\gamma_0$ are pure quaternions, so
$\gamma'=-\gamma$ and $\gamma_0'=-\gamma_0$. In particular,
$(\gamma\gamma_0)'=\gamma_0'\gamma'=\gamma_0\gamma$. On the other
hand, using (\ref{eq-1}),
$$
(\gamma\gamma_0)'=(\gamma_0(a+b\gamma))'=(a+b\gamma)'\gamma_0' =
(a-b\gamma)(-\gamma_0)=-a\gamma_0+b\gamma\gamma_0.
$$
Using (\ref{eq-1}) again, this last expression is equal to
$$
-a\gamma_0+b\gamma_0(a+b\gamma)=-a\gamma_0+ab\gamma_0+b^2\gamma_0\gamma.
$$
Overall,
$$
\gamma_0\gamma=(ab-a)\gamma_0+b^2\gamma_0\gamma.
$$
Multiplying both sides by $\gamma_0^{-1}$ from the left, we conclude
that either $b=-1$ and $a=0$, or $b=1$ and $a\in \F_q$. In the first
case, (\ref{eq-1}) becomes $\gamma_0\gamma=-\gamma\gamma_0$, so
$\bD\cong H(\xi, \xi \wp_o)$. Now, since both $\xi$ and $\xi \wp_o$
are units at the places in $R-o$, the Hilbert symbols $(\xi, \xi
\wp_o)_v$ for $v\in R-o$ are equal to $1$; see Corollary to
Proposition XIV.8 in \cite{SerreLF}. This implies that $\bD$ is
split at the places in $R-o$ (see \cite[p. 32]{Vigneras}), which is
a contradiction. Therefore, we must have $b=1$ and
$$
\gamma\gamma_0=\gamma_0(a+\gamma).
$$
Multiplying both sides of this equality by $\gamma_0\gamma$ from the
right, we get $\xi^2 \wp_o$ on the left hand-side and
$$
a\gamma_0^2\gamma+\gamma_0\gamma\gamma_0\gamma =a\xi
\wp_o\gamma+\gamma_0\gamma_0(a+\gamma)\gamma=a\xi \wp_o\gamma+\xi
\wp_o(a\gamma +\xi).
$$
on the right hand-side. This forces $2a\xi \wp_o=0$, so $a=0$. But
if $a=0$ and $b=1$, then (\ref{eq-1}) says that $\gamma$ and
$\gamma_0$ commute in $\bD$. Therefore, $F(\gamma, \gamma_0)$ is a
subfield of $\bD$. This subfield is necessarily quadratic over $F$.
On the other hand, $F(\gamma_0)$ and $F(\gamma)$ are both quadratic
and linearly disjoint over $F$. This leads to a contradiction.
Overall, we conclude that the conditions (\ref{eq-4}) cannot hold
simultaneously, which finishes the proof of Part (3) of Theorem
\ref{thm3.1}.


\section{The place at infinity}\label{SecGU}
As in the introduction, let $\La$ be a maximal $A$-order in $D$.
From now on we denote $K:=\Fi$, $\cO:=\cO_\infty$,
$k:=\F_\infty\cong\F_q$, and $\G:=\La^\times$. The group $\G$ can be
considered as a discrete subgroup of $\GL_2(K)$ via an embedding
$$
\iota:\G\hookrightarrow D^\times(F)\hookrightarrow D^\times(K)\cong
\GL_2(K).
$$
Note that for $\gamma\in \G$, $\det(\iota(\gamma))=\Nr(\gamma)\in
\F_q^\times$, so $\ord_\infty\det(\iota(\gamma))=0$. We fix some
embedding $\iota$ and omit it from notation. $\G$ acts on Drinfeld's
$\infty$-adic half-plane $\Omega:=\P^{1,\an}_{K}-\P^{1,\an}_{K}(K)$
and its associated formal scheme $\widehat{\Omega}$. Let
$M:=X^\cD\times_C \Spec(\cO)$, and denote by $\widehat{M}$ the
completion of $M$ along its special fibre. We have the following
uniformization theorem due to Blum and Stuhler \cite[Thm.
4.4.11]{BS}
\begin{equation}\label{eq-unif1}
\widehat{M} \cong D^\times(F)\bs [\widehat{\Omega}\times
(D^\times(\A^\infty)/(\cD^\infty)^\times)].
\end{equation}
Since $D$ is split at $\infty$, the Strong Approximation Theorem for
$D^\times$ implies (cf. \cite[p. 89]{Vigneras})
\begin{equation}\label{eq-sat}
D^\times(F)\bs D^\times(\A^\infty)/(\cD^\infty)^\times\cong
F^\times\bs (\A^\infty)^\times/\prod_{x\in
|C|-\infty}\cO_x^\times\cong \Pic(A)=1.
\end{equation}
Note that $\G\cong D^\times(F) \cap (\cD^\infty)^\times$. Therefore,
(\ref{eq-unif1}) reduces to
\begin{equation}\label{eq-unifinf}
\widehat{M} \cong \G\bs\widehat{\Omega}.
\end{equation}
The uniformization (\ref{eq-unifinf}) will play a key role in our
examination of rational points on $X^\cD$ over finite extensions of
$K$. We start with an analysis of the quotient graph $\G\bs \cT$,
which can be considered as an analogue of fundamental domains of
Shimura curves (here $\cT$ is the Bruhat-Tits tree of $\PGL_2(K)$).
By Lemma \ref{lem1.6}, $\G$ acts without inversion on $\cT$ and, by
Theorem \ref{thm-Kurihara}, $\G\bs \cT$ is the dual graph of $M$; in
particular, $\G\bs \cT$ is a finite graph. This last fact can easily
be proved directly, without an appeal to the uniformization theorem:

\begin{lem}\label{propFG}
The quotient graph $\G\bs \cT$ is finite.
\end{lem}
\begin{proof}
It is enough to show that $\G\bs \cT$ has finitely many vertices.
The group $\GL_2(K)$ acts transitively on the set of lattices in
$K^2$. The stabilizer of a vertex is isomorphic to $Z(K)\cdot
\GL_2(\cO)$, where $Z$ denotes the center of $\GL(2)$. This yields a
natural bijection $\Ver(\cT)\cong \GL_2(K)/Z(K)\cdot \GL_2(\cO)$,
and
$$
\Ver(\G\bs \cT)\cong \G\bs \GL_2(K)/Z(K)\cdot \GL_2(\cO).
$$
We will show that the above double coset space is finite. Since $D$
is a division algebra, $D^\times(F)\bs D^\times(\A)/Z(K)$ is
compact, cf. \cite[Ch. III.1]{Vigneras}. Thus, using (\ref{eq-sat}),
we see that $\G\bs \GL_2(K)/Z(K)$ is compact since it is the image
of $D^\times(F)\bs D^\times(\A)/Z(K)$ under the natural quotient map
$D^\times(\A)\to D^\times(\A)/(\cD^\infty)^\times$. Finally, since
$\GL_2(\cO)$ is open in $\GL_2(K)$, $\#\Ver(\G\bs \cT)$ is finite.
\end{proof}

\begin{prop}\label{Prop2.1}
Let $v\in \Ver(\cT)$ and $y\in \Ed(\cT)$. Then $\G_v\cong
\F_q^\times$ or $\G_v\cong \F_{q^2}^\times$, and $\G_y\cong
\F_q^\times$.
\end{prop}
\begin{proof}
By choosing an appropriate basis of $K^2$, we can assume that $v$
represents the homothety class of $\cO^2$. The stabilizer of $v$ in
$\GL_2(K)$ is $Z(K)\cdot \GL_2(\cO)$. Since $\GL_2(\cO)$ is compact
in $\GL_2(K)$, whereas $\G$ is discrete, $\G_v=\G\cap \GL_2(\cO)$ is
finite. In particular, if $\gamma\in \G_v$, then $\gamma^n=1$ for
some $n\geq 1$. We claim that the order $n$ of $\gamma$ is coprime
to the characteristic $p$ of $F$. Indeed, if $p|n$ then
$(\gamma^{n/p}-1)\in D$ is non-zero but $(\gamma^{n/p}-1)^p=0$. This
is not possible since $D$ is a division algebra. Consider the
subfield $F(\gamma)$ of $D$ generated by $\gamma$ over $F$. By
Proposition \ref{prop-emb}, $[F(\gamma):F]=1$ or $2$. Since
$\gamma\in D$ is algebraic over $\F_q$, we conclude that
$[\F_q(\gamma):\F_q]=1$ or $2$.

It is obvious that $\F_q^\times\subset \G_v$. Assume there is
$\gamma\in \G_v$ which is not in $\F_q^\times$. From the previous
paragraph, $\gamma$ generates $\F_{q^2}$ over $\F_q$. Considering
$\gamma$ as an element of $\GL_2(\cO)$, we clearly have $a+b
\gamma\in \M_2(\cO)$ for $a, b\in \F_q$ (embedded diagonally into
$\GL_2(K)$). But if $a$ and $b$ are not both zero, then
$a+b\gamma\in \La$ is invertible, hence belongs to $\G$ and
$\GL_2(\cO)$. We conclude that
$\F_q(\gamma)^\times\cong\F_{q^2}^\times\subset \G_v$, and moreover,
every element of $\G_v$ is of order dividing $q^2-1$. In particular,
by Sylow theorem the order of $\G_v$ is coprime to $p$. Suppose
there is $\delta\in \G_v$ which is not in $\F_q(\gamma)^\times$.
Since $\delta$ is algebraic over $\F_q$, $\delta$ and $\gamma$ do
not commute in $D$ (otherwise $F(\gamma, \delta)$ is a subfield of
$D$ of degree $>2$ over $F$). Then $\G_v/\F_q^\times$ is a finite
subgroup of $\PGL_2(K)$ whose elements have orders dividing $q+1$
and which contains two non-commuting elements of order $(q+1)$. This
contradicts the classification of finite subgroups of $\PGL_2(K)$ of
order coprime to $p$, cf. \cite[p. 281]{SerreG}. Indeed, suppose $H$
is such a subgroup. If $p=2$ or $3$, then $H$ is either cyclic or
dihedral. In general, if $H$ is not cyclic or dihedral, then it must
be isomorphic to $A_4$, $A_5$ or $S_5$, so the elements of $H$ have
orders $1,2,3,4$ or $5$.

Now consider $\G_y$. Clearly $\F_q^\times\subset \G_y$. Let $v$ and
$w$ be the extremities of $y$. Note that there are natural
inclusions $\G_y\subset \G_v$, $\G_y\subset \G_w$ and $\G_y=\G_v\cap
\G_w$. If $\G_y$ is strictly larger than $\F_q^\times$, then from
the discussion about the stabilizers of vertices, we have
$\G_v=\G_w\cong \F_{q^2}^\times$ (an equality of subgroups of $\G$).
Therefore, $\G_y\cong \F_{q^2}^\times$. On the other hand, the
stabilizer of $y$ in $\GL_2(K)$ is isomorphic to $Z(K)\cdot \cI$,
where $\cI$ is the Iwahori
subgroup $\cI$ of $\GL_2(\cO)$ consisting of matrices $\begin{pmatrix} a & b \\
c & d
\end{pmatrix}$ such that $\ord_\infty(c)\geq 1$. Since $\cI$ does
not contain a subgroup isomorphic to $\F_{q^2}^\times$, we get a
contradiction.
\end{proof}

\begin{cor}\label{cor2.2} Let $v\in \Ver(\cT)$ be such that
$\G_v\cong \F_{q^2}^\times$. Then $\G_v$ acts transitively on the
edges with origin $v$.
\end{cor}
\begin{proof}
By Proposition \ref{Prop2.1}, a subgroup of $\F_{q^2}^\times$ which
stabilizes an edge with origin $v$ is $\F_q^\times$. Hence
$\G_v/\F_q^\times$ acts freely on the set of such edges. Since this
quotient group has $q+1$ elements, which is also the number of edges
with origin $v$, it has to act transitively.
\end{proof}

\begin{cor}\label{cor5.4} The length of any edge $y\in \Ed(\G\bs \cT)$ is equal to
$1$. In particular, $M$ is a proper, flat and regular scheme over
$\Spec(\cO)$ with $\Gr(M)\cong \G\bs \cT$. The first Betti number
$h_1(\G\bs \cT)$ is equal to
$$
g(R):=1+\frac{1}{q^2-1}\prod_{x\in R}(q_x-1)-\frac{q}{q+1}\cdot
2^{\# R-1}\cdot \Odd(R).
$$
$($$M$ is not necessarily the minimal regular model of $X^\cD$ over
$\Spec(\cO)$; one obtains the minimal model by removing the terminal
vertices from $\G\bs\cT$.$)$
\end{cor}
\begin{proof} Let $y \in \Ed(\G\bs \cT)$.
According to Proposition \ref{Prop2.1}, the image of $\G_y$ in
$\PGL_2(K)$ is trivial, so $\ell(y)=1$ by definition. Next, by
Theorem \ref{thm-Kurihara}, $h_1(\G\bs\cT)$ is equal to the genus of
$X^\cD_K$. A formula for the genus of $X^\cD_K$ is computed in
\cite[Thm. 5.4]{PapGenus} and it is given by $g(R)$, so $h_1(\G\bs
\cT)=g(R)$.
\end{proof}

\begin{thm}\label{PropTV}\hfill
\begin{enumerate}
\item The graph $\G\bs \cT$ has no loops.
\item Every vertex of $\G\bs \cT$ is either terminal or has degree
$q+1$.
\item Let $V_1$ and $V_{q+1}$ be the number of terminal and
degree $q+1$ vertices of $\G\bs \cT$, respectively. Then
$$
V_1=2^{\# R -1}\Odd(R)\quad \text{and}\quad
V_{q+1}=\frac{2}{q-1}\left(g(R)-1+2^{\# R -2}\Odd(R)\right).
$$
\end{enumerate}
\end{thm}
\begin{proof} The graph $\G\bs \cT$ has no loops since adjacent
vertices of $\cT$ are not $\G$-equivalent, as follows from Lemma
\ref{lem1.6}.

Let $v\in \Ver(\cT)$. By Proposition \ref{Prop2.1}, $\G_v\cong
\F_q^\times$ or $\F_{q^2}^\times$. In the second case, by Corollary
\ref{cor2.2}, the image of $v$ in $\G\bs \cT$ is a terminal vertex.
Now assume $\G_v\cong \F_q^\times$. We claim that the image of $v$
in $\G\bs \cT$ has degree $q+1$. Let $e, y\in \Ed(\cT)$ be two
distinct edges with origin $v$. It is enough to show that $e$ is not
$\G$-equivalent to $y$ or $\bar{y}$. On the one hand, $e$ cannot be
$\G$-equivalent to $y$ since $\G_v\cong \F_q^\times$ stabilizes
every edge with origin $v$. On the other hand, if $e$ is
$\G$-equivalent to $\bar{y}$ then $v$ is $\G$-equivalent to an
adjacent vertex, and that cannot happen.

Using (2), the number of non-oriented edges $\{y, \bar{y}\}$ of
$\G\bs\cT$ is equal to
$$
E:=(V_1+(q+1)V_{q+1})/2.
$$
On the other hand, by Euler's formula $E+1=g(R)+V_1+V_{q+1}$, so to
prove (3) it is enough to prove the formula for $V_1$.

Let $S$ be the set of terminal vertices of $\G\bs \cT$. Let $G$ be
the set of conjugacy classes of subgroups of $\G$ isomorphic to
$\F_{q^2}^\times$. We claim that there is a bijection $\varphi:S\to
G$ given by $\tilde{v}\mapsto \G_v$, where $v$ is a preimage of the
terminal vertex $\tilde{v}\in S$. The map is well-defined since if
$w$ is another preimage of $\tilde{v}$ then $v=\gamma w$ for some
$\gamma\in \G$, and so $\G_w=\gamma^{-1} \G_v\gamma$ is a conjugate
of $\G_v$. If $\varphi$ is not injective, then there are two
vertices $v,w\in\Ver(\cT)$ such that $\G_v\cong\F_{q^2}^\times$,
$\G_w=\gamma^{-1}\G_v\gamma$ for some $\gamma\in \G$, but $v$ and
$w$ are not in the same $\G$-orbit. Then $\G_{\gamma
w}=\gamma\G_w\gamma^{-1}=\G_v$, but $\gamma w\neq v$. The geodesic
connecting $v$ to $\gamma w$ is fixed by $\G_v$, so every edge on
this geodesic has stabilizer equal to $\G_v$. This contradicts
Proposition \ref{Prop2.1}. Finally, to see that $\varphi$ is
surjective it is enough to show that every finite subgroup $\G'$ of
$\G$ fixes some vertex. Since $\G'$ is finite, the orbit $\G'\cdot
v$ is finite for any $v\in \Ver(\cT)$, so $\G'$ fixes some vertex by
\cite[Prop. 19, p. 36]{SerreT}.

Let $\cA:=\F_{q^2}[T]$ and $L:=\F_{q^2}F$ (note that $\cA$ is the
integral closure of $A$ in $L$). If $q$ is odd, let $\xi$ be a fixed
non-square in $\F_q$. If $q$ is even, let $\xi$ be a fixed element
of $\F_q$ such that $\Tr_{\F_q/\F_2}(\xi)=1$ (such $\xi$ always
exists, cf. \cite[Thm. 2.24]{LN}). Consider the polynomial
$f(x)=x^2-\xi$ if $q$ is odd, and $f(x)=x^2+x+\xi$ if $q$ is even.
Note that $f(x)$ is irreducible over $\F_q$; this is obvious for $q$
odd, and follows from \cite[Cor. 3.79]{LN} for $q$ even. Thus, a
solution of $f(x)=0$ generates $\F_{q^2}$ over $\F_q$. Denote by $P$
the set of $\G$-conjugacy classes of elements of $\La$ whose reduced
characteristic polynomial is $f(x)$. By a theorem of Eichler (Cor.
5.12, 5.13, 5.14 on pp. 94-96 of \cite{Vigneras}),
$$
\# P = h(\cA)\prod_{x\in R}\left(1-\left(\frac{L}{x}\right)\right),
$$
where $h(\cA)$ is the class number of $\cA$, and
$\left(\frac{L}{x}\right)$ is the \textit{Artin-Legendre symbol}:
$$
\left(\frac{L}{x}\right)=\left\{
                           \begin{array}{ll}
                             1, & \hbox{if $x$ splits in $L$;} \\
                             -1, & \hbox{if $x$ is inert in $L$;} \\
                             0, & \hbox{if $x$ ramifies in $L$.}
                           \end{array}
                         \right.
$$
Note that a place of even degree splits in $L$ and a place of odd
degree remains inert in $L$, so
$$
\prod_{x\in R}\left(1-\left(\frac{L}{x}\right)\right)= 2^{\#
R}\Odd(R).
$$
Since $h(\cA)=1$, we get $\# P = 2^{\# R}\Odd(R)$.

If $\la\in \La$ is an element with reduced characteristic polynomial
$f(x)$, then it is clear that $\la\in \G$ and
$\F_q(\la)^\times\subset \G$ is isomorphic to $\F_{q^2}^\times$.
Hence $\la\mapsto \F_q(\la)^\times$ defines a map $\chi:P\to G$. As
before, it is easy to check that $\chi$ is well-defined and
surjective. We will show that $\chi$ is 2-to-1, which implies the
formula for $V_1$. Obviously $\la\neq \la'$ and these elements
generate the same subgroup in $\G$, as the canonical involution on
$D$ restricted to $F(\la)$ is equal to the Galois conjugation on
$F(\la)/F$. Since $\la$ and $\la'$ are the only elements in
$\F_q(\la)$ with the given characteristic polynomial, it is enough
to show that $\la$ and $\la'$ are not $\G$-conjugate. Suppose there
is $\gamma\in \G$ such that $\la'=\gamma \la \gamma^{-1}$. One
easily checks that $1, \la, \gamma, \gamma\la$ are linearly
independent over $F$, hence generate $D$. If $q$ is odd, then
$\la'=-\la$. If $q$ is even, then $\la'=\la+1$. Using this, one
easily checks that $\gamma^2$ commutes with $\la$, e.g., for $q$
odd:
$$
\gamma^2 \la \gamma^{-2}=\gamma \la'\gamma^{-1}=-\gamma \la
\gamma^{-1}=-\la'=\la.
$$
Hence $\gamma^2$ lies in the center of $D$, and therefore,
$\gamma^2=b\in \F_q^\times$. Looking at the relations between $\la$
and $\gamma$, we see that $D$ is isomorphic to $H(\xi, b)$. We claim
that this last algebra is isomorphic to $\M_2(F)$, which leads to a
contradiction. Indeed, an easy consequence of Chevalley-Warning
theorem is that a quadratic form in $n\geq 3$ variables over a
finite field has a non-trivial zero. Thus, since $\xi, b\in
\F_q^\times$, the quadratic form associated to the reduced norm on
$H(\xi,b)$ has a non-trivial zero over the subfield $\F_q$ of $F$.
This obviously implies that $H(\xi,b)$ has zero divisors, hence
cannot be a division algebra.
\end{proof}

\begin{rem}
It is easy to check directly that the formulas giving
$h_1(\G\bs\cT)$ and $V_{q+1}$ assume non-negative integer values.
Interestingly, the presence of $\Odd(R)$ is necessary to make this
happen.
\end{rem}

The next theorem is the group-theoretic incarnation of Theorem
\ref{PropTV}.

\begin{thm}\label{thmGTI} Let $\G_\tor$ be the normal subgroup of $\G$ generated
by torsion elements.
\begin{enumerate}
\item $\G/\G_\tor$ is a free group on $g(R)$ generators.
\item If $\Odd(R)=0$, then $\G_\tor=\F_q^\times$.
\item If $\Odd(R)=1$, then the maximal finite order subgroups of $\G$
are isomorphic to $\F_{q^2}^\times$, and, up to conjugation, $\G$
has $2^{\# R-1}$ such subgroups.
\item $\G$ can be generated by $2^{\# R-1}+g(R)$ elements.
\end{enumerate}
\end{thm}
\begin{proof}
By \cite[Cor. 1, p. 55]{SerreT}, $\G/\G_\tor$ is the fundamental
group of the graph $\G\bs\cT$. The topological fundamental group of
any finite graph is a free group. Hence $\G/\G_\tor$ is a free
group. The number of generators of this group is equal to the number
of generators of the commutator group $(\G/\G_\tor)^\mathrm{ab}\cong
H_1(\G\bs\cT, \Q)$, which is a free abelian group on $g(R)$
generators. This proves (1). Parts (2) and (3) follow from the
proofs of Proposition \ref{Prop2.1} and Theorem \ref{PropTV}. Part
(4) is a consequence of (1)-(3), cf. \cite[Thm. 13, p. 55]{SerreT}.
\end{proof}

The main point of \cite[Ch. I]{SerreT} is that the knowledge of
$\G\bs\cT$ gives a presentation for $\G$. We apply this theory in
the case when $\G\bs \cT$ is itself a tree:

\begin{cor}\label{thm-tree}
$\G=\G_\tor$ if and only if one of the following holds:
\begin{enumerate}
\item $R=\{x,y\}$ and $\deg(x)=\deg(y)=1$. In this case, $\G$ has a
presentation
$$
\G\cong \langle \gamma_1,\gamma_2\ |\
\gamma_1^{q^2-1}=\gamma_2^{q^2-1}=1,\
\gamma_1^{q+1}=\gamma_2^{q+1}\rangle.
$$
\item $q=4$ and $R$ consists of the
four degree-$1$ places in $|C|-\infty$. In this case, $\G$ has a
presentation
$$
\G\cong \langle \gamma_1,\dots,\gamma_8\ |\ \gamma_1^{15}=\cdots
=\gamma_8^{15}=1,\ \gamma_1^{5}=\cdots=\gamma_8^{5}\rangle.
$$
\end{enumerate}
\end{cor}
\begin{proof} By Theorem \ref{thmGTI},
$\G=\G_\tor$ if and only if $g(R)=0$. From the formula for $g(R)$
one easily concludes that $g(R)=0$ exactly in the two cases listed
in the theorem. In Case (1), according to Theorem \ref{PropTV},
$V_1=2$ and $V_{q+1}=0$, so $\G\bs \cT$ is a segment:
\begin{center}
\begin{picture}(30,10)
\put(5,5){\circle*{2}}\put(5,5){\line(1,0){20}}
\put(25,5){\circle*{2}}
\end{picture}
\end{center}
In Case (2), $V_1=8$ and $V_{5}=2$, so $\G\bs \cT$ is the tree:
\begin{center}
\begin{picture}(60,40)
\put(5,5){\circle*{2}}\put(5,5){\line(1,1){15}}
\put(5,15){\circle*{2}}\put(5,15){\line(3,1){15}}
\put(5,25){\circle*{2}}\put(5,25){\line(3,-1){15}}
\put(5,35){\circle*{2}}\put(5,35){\line(1,-1){15}}
\put(20,20){\circle*{2}}\put(20,20){\line(1,0){20}}
\put(40,20){\circle*{2}}
\put(55,5){\circle*{2}}\put(55,5){\line(-1,1){15}}
\put(55,15){\circle*{2}}\put(55,15){\line(-3,1){15}}
\put(55,25){\circle*{2}}\put(55,25){\line(-3,-1){15}}
\put(55,35){\circle*{2}}\put(55,35){\line(-1,-1){15}}
\end{picture}
\end{center}
By \cite[I.4.4]{SerreT}, $\G$ is the graph of groups $\G\bs \cT$,
where each terminal vertex of $\G\bs \cT$ is labeled by
$\F_{q^2}^\times$, each non-terminal vertex is labeled by
$\F_q^\times$, and each edge is labeled by $\F_q^\times$. In other
words, $\G$ is the amalgam of the groups labeling the vertices of
$\G\bs\cT$ along the subgroups labeling the edges. The presentation
for $\G$ follows from the definition of amalgam; see
\cite[I.1]{SerreT}.
\end{proof}

Theorem \ref{PropTV} allows to determine $\G\bs \cT$ in some other
cases, besides the case when $\G\bs \cT$ is a tree treated in
Corollary \ref{thm-tree}:

\begin{cor}\label{thm-hyp}
Suppose $R=\{x, y\}$ and $\{\deg(x), \deg(y)\}=\{1,2\}$. Then $\G\bs
\cT$ is the graph which has $2$ vertices and $q+1$ edges connecting
them:
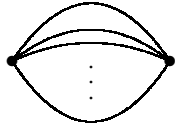
\begin{center}
\begin{picture}(40,25)
\qbezier(5,13)(20,35)(35,13)\qbezier(5,13)(20,25)(35,13)
\qbezier(5,13)(20,20)(35,13) \qbezier(5,13)(20,-10)(35,13)
\put(5,13){\circle*{2}}\put(35,13){\circle*{2}}
\put(20,12){\circle*{.7}}\put(20,9){\circle*{.7}}\put(20,6){\circle*{.7}}
\end{picture}
\end{center}
\end{cor}
\begin{proof}
From Theorem \ref{PropTV}, $h_1(\G\bs\cT)=q$, $V_1=0$ and
$V_{q+1}=2$. This implies the claim.
\end{proof}

The example in Corollary \ref{thm-hyp} is significant for arithmetic
reasons. Assume $q$ is odd. As is shown in \cite{PapHD}, the curve
$X^\cD$ is hyperelliptic if and only if $R=\{x, y\}$ and $\{\deg(x),
\deg(y)\}=\{1,2\}$. Thus, Corollaries \ref{thm-hyp} and \ref{cor5.4}
imply that the special fibre of the minimal regular model over
$\Spec(\cO)$ of a hyperelliptic $X^\cD$ consists of two projective
lines $\P^1_{\F_q}$ intersecting transversally at their
$\F_q$-rational points.

\begin{thm}\label{thmLPinf} $M(K)=\emptyset$ if and only if $\Odd(R)=0$.
If $L$ is a finite non-trivial extension of $K$, then $M(L)\neq
\emptyset$.
\end{thm}
\begin{proof} Let $L$ be a finite extension of $K$ with relative degree
$f$ and ramification index $e$. Let $\cO_L$ be the ring of integers
of $L$. The argument which proves \cite[Thm. 4.5]{JL}, combined with
Corollary \ref{cor5.4}, shows that the dual graph of $M\times_\cO
\cO_L$ is isomorphic to $\G\bs\cT$ but with $\ell(y)=e$ for any
$y\in \Ed(\G\bs\cT)$. Moreover, the action induced by $\Fr_\infty$
on this graph is trivial. Using a geometric version of Hensel's
lemma \cite[Lem. 1.1]{JL}, one deduces from the previous statement
that $M(L)\neq \emptyset$ if and only if either $e>1$ or there is a
vertex $x\in \Ver(\G\bs\cT)$ whose degree $< q^f+1$; cf. the
argument in the proofs \cite[Thm. 5.1, 5.2]{JL}. But by Theorem
\ref{PropTV} every vertex of $\G\bs\cT$ has degree $1$ or $q+1$, and
there are vertices of degree $1$ if and only if $\Odd(R)=1$. Since
$q+1< q^f+1$ if $f>1$, the claim follows.
\end{proof}


\section{Explicit generators and equations}\label{SecExplU}
Let $B$ be an indefinite division quaternion algebra over $\Q$ of
discriminant $d$ and let $\La\subset B$ be a maximal $\Z$-order;
recall that $d>1$ is the product of primes where $B$ ramifies. Let
$\G^d=\{\gamma\in \La\ |\ \Nr(\gamma)=1\}$. Upon fixing an
identification of $B\otimes_\Q \R$ with $\M_2(\R)$, we can view the
group $\G^d$ as a discrete subgroup of $\SL_2(\R)$. Only for a few
$d$ the explicit matrices generating $\G^d$ as a subgroup of
$\SL_2(\R)$ are known, cf. \cite{AB} or \cite{KV}. For example,
$\G^6$ is isomorphic to the subgroup of $\SL_2(\R)$ generated by
\begin{align*}
\gamma_1=\frac{1}{2}\begin{pmatrix} \sqrt{2} & 2-\sqrt{2} \\
-6-3\sqrt{2} & -\sqrt{2}\end{pmatrix}&,\quad
\gamma_2=\frac{1}{2}\begin{pmatrix} \sqrt{2} & -2+\sqrt{2} \\
6+3\sqrt{2} & -\sqrt{2}\end{pmatrix}\\
\gamma_3=\frac{1}{2}\begin{pmatrix} 1 & 1 \\
-3 & 1\end{pmatrix}&,\quad
\gamma_4=\frac{1}{2}\begin{pmatrix} 1 & 3-2\sqrt{2} \\
-9-6\sqrt{2} & 1\end{pmatrix}
\end{align*}
(which have orders $4$, $4$, $6$, $6$, respectively).

Let $\cH$ be the Poincar\'e upper half-plane. The quotient
$X^d:=\G^d\bs\cH$ is a compact Riemann surface. From the work of
Shimura it is known that the algebraic curve $X^d$ has a canonical
model over $\Q$. The equations defining $X^d$ as a curve over $\Q$
again are known only for a small number of $d$; cf. \cite{Kurihara},
\cite{JL}. For example, $X^6$ as a curve in $\P^2_\Q$ is isomorphic
to the conic defined by $$ X^2+Y^2+3Z^2=0.$$

\vspace{0.1in}

In this section we keep the notation of $\S$\ref{SecGU}, but assume
$q$ is odd. We will find explicit generators for $\G$ and determine
the equation of $X^\cD$ in Case (1) of Corollary \ref{thm-tree}.
First, we explicitly describe $D$ and a maximal $A$-order in $D$:

\begin{lem}\label{lem6.1} Let $\xi\in \F_{q}$ be a fixed non-square.
If $\Odd(R)=1$, then $H(\xi,\fr)\cong D$. The free $A$-module $\La$
in $D$ generated by
$$
x_1=1,\quad x_2=i,\quad x_3=j,\quad x_4=ij
$$
is a maximal order.
\end{lem}
\begin{proof} To prove $H(\xi,\fr)\cong D$,
it is enough to show that the Hilbert symbol $(\xi,\fr)_v$ is $-1$
if and only if $v\in R$, cf. \cite[p. 32]{Vigneras}. By \cite[p.
217]{SerreLF}, $(\xi,\fr)_v=1$ if and only if $\xi^{\ord_v(\fr)}$ is
a square in $\F_v$. Now $\ord_v(\fr)=0$ if $v\in |C|-R-\infty$,
$\ord_v(\fr)=1$ if $v\in R$, $\ord_\infty(\fr)$ is even since $\# R$
is even and $\Odd(R)=1$. It remains to observe that $\xi$ is not a
square in $\F_v$ for $v\in R$ since $\xi$ is not a square in $\F_q$
and $\deg(v)$ is odd by assumption.

Next, it is obvious that $\La$ is an order. To show that it is
maximal, we compute its discriminant, i.e., the ideal of $A$
generated by $\det(\Tr(x_ix_j))_{ij}$:
$$
\det(\Tr(x_ix_j))_{ij}=\det \begin{pmatrix} 2 & 0 & 0 & 0\\
0 & 2\xi & 0 & 0\\
0 & 0 & 2\fr & 0\\
0 & 0 & 0 & -2\xi \fr
\end{pmatrix}=-16\xi^2\fr^2.
$$
Since $q$ is odd and $\xi\in \F_q^\times$, the ideal generated by
$-16\xi^2\fr^2$ is $(\prod_{x\in R}\fp_x)^2$. This implies that
$\La$ is maximal, cf. \cite[pp. 84-85]{Vigneras}.
\end{proof}

From Lemma \ref{lem6.1} we see that finding the elements
$\la=a+bi+cj+dij\in D$ which lie in $\G$ is equivalent to finding
$a,b,c,d,\in A$ such that
\begin{equation}\label{eq-GNr}
(a^2-\xi b^2)-\fr(c^2-\xi d^2) \in \F_q^\times.
\end{equation}
We are particularly interested in torsion elements of $\G$, so
instead of looking for general units, we will try to find elements
in $\La$ which are algebraic over $\F_q$ (such elements
automatically lie in $\La^\times$). Consider the equation
$\gamma^2=\xi$ in $\La$. If we write $\gamma=a+bi+cj+dij$, then
$$
\gamma^2= a^2+b^2\xi +c^2\fr-d^2\xi\fr+2(abi+acj+adij).
$$
Therefore, $\gamma^2=\xi$ is equivalent to
\begin{equation}\label{eq-tuodd}
a=0\quad \text{and} \quad b^2\xi  +c^2\fr-d^2\xi\fr=\xi.
\end{equation}
A possible solution is $b=1$, $d=c=0$. This gives the obvious
$\theta_1=i$ as a torsion unit. Now assume $R=\{x,y\}$ with
$\deg(x)=\deg(y)=1$. Then $\fr=(T-\alpha_1)(T-\alpha_2)$, where
$\alpha_1,\alpha_2\in \F_q$ and $\alpha_1\neq \alpha_2$. For this
$\fr$
$$
b_0=\frac{2}{\alpha_1-\alpha_2}
T-\frac{\alpha_1+\alpha_2}{\alpha_1-\alpha_2},\quad c_0=0, \quad
d_0=-\frac{2}{\alpha_1-\alpha_2}
$$
satisfy (\ref{eq-tuodd}), so $\theta_2=i(b_0+d_0j)$ is a torsion
unit.

Next, we study the action of $\theta_1, \theta_2$ on $\cT$. Since
$\fr$ has even degree and is monic, $\sqrt{\fr}\in K$. The map
\begin{equation}\label{eq-emb}
i\mapsto \begin{pmatrix} 0 & 1\\ \xi & 0 \end{pmatrix}, \quad
j\mapsto
\begin{pmatrix} \sqrt{\fr} & 0\\ 0 & -\sqrt{\fr}\end{pmatrix},
\end{equation}
defines an embedding of $D$ into $\M_2(K)$. The $1/T$-adic expansion
of $\sqrt{\fr}$ in $K$ starts with
$$
\sqrt{\fr}=T-\frac{\alpha_1+\alpha_2}{2}-\frac{(\alpha_1-\alpha_2)^2}{8T}+\cdots
$$
Therefore,
$$
\pi:=b_0+d_0\sqrt{\fr}=\frac{\alpha_1-\alpha_2}{4T}+\cdots
$$
has valuation $\ord_\infty(\pi)=1$. To simplify the notation in our
calculations we take $\pi$ as a uniformizer of $\cO$. Note that
$\pi^{-1}=b_0-d_0\sqrt{\fr}$, so $\theta_2$ under the embedding
(\ref{eq-emb}) maps to the matrix
$$
\theta_2=\begin{pmatrix} 0 & 1\\ \xi & 0
\end{pmatrix} \begin{pmatrix} \pi  & 0 \\ 0 & \pi^{-1}
\end{pmatrix}=\begin{pmatrix} 0 & \pi^{-1}\\ \xi\pi & 0
\end{pmatrix}.
$$
Let $e_1=\begin{pmatrix} 1\\ 0\end{pmatrix}$ and
$e_2=\begin{pmatrix} 0\\ 1\end{pmatrix}$ be the standard basis of
$K^2$. Let $v, w\in \Ver(\cT)$ be the vertices corresponding to the
homothety classes of lattices $\cO e_1 + \cO e_2$ and $\cO e_1 + \pi
\cO e_2$, respectively. Note that $v$ and $w$ are adjacent in $\cT$.
Now $\theta_1 e_1 = \xi e_2$ and $\theta_1 e_2=e_1$, so
$$
\theta_1 \cdot v = [\xi\cO e_2 + \cO e_1] = [\cO e_1+\cO e_2 ]=v.
$$
Similarly, $\theta_2 e_1=\xi \pi e_2$ and $\theta_2 e_2=e_1/\pi$, so
$$
\theta_2\cdot w=[\xi \pi \cO e_2 + \cO e_1]=[\cO e_1 + \pi\cO
e_2]=w.
$$
Thus, we found two adjacent vertices in $\cT$ and elements in their
stabilizers which are not in $\F_q^\times$. Proposition
\ref{Prop2.1} implies that $\G_v\cong\F_{q^2}^\times$ and
$\G_w\cong\F_{q^2}^\times$. By Corollary \ref{cor2.2}, the images of
$v$ and $w$ in $\G\bs \cT$ are adjacent terminal vertices. This
implies that $\G\bs\cT$ is an edge, and proves Theorem \ref{PropTV}
in the case when $R=\{x, y\}$ and $\deg(x)=\deg(y)=1$. By fixing
generators $\gamma_1, \gamma_2$ of the finite cyclic groups
$\F_q(\theta_1)^\times=\langle\gamma_1\rangle$ and
$\F_q(\theta_2)^\times=\langle\gamma_2\rangle$, one obtains two
torsion elements which generate $\G$.

\begin{example}
Let $q=3$, $\alpha_1=1$, $\alpha_2=0$. Then $\fr=T(T-1)$, $\xi=-1$.
Since $\gamma_i=1-\theta_i$ generates $\F_q(\theta_i)^\times$, $\G$
is isomorphic to the subgroup of $\GL_2(K)$ generated by the
matrices
$$
\gamma_1=\begin{pmatrix} 1 & -1 \\ 1 & 1\end{pmatrix}\quad
\text{and} \quad \gamma_2=\begin{pmatrix} 1 & (T+1)-\sqrt{\fr} \\
-(T+1)-\sqrt{\fr} & 1\end{pmatrix}
$$
both of which have order $8$ and satisfy $\gamma_1^4=\gamma_2^4=-1$.
\end{example}

\begin{thm}\label{thmEquations}
Assume $R=\{x, y\}$ and $\deg(x)=\deg(y)=1$. Then $X^\cD_F$, as a
curve over $F$, is isomorphic to the conic in $\P^2_F$ defined by
the equation
$$
X^2-\xi Y^2 -\fr Z^2 =0.
$$
\end{thm}
\begin{proof} Let $\cC$ be a smooth conic in $\P^2_F$, and
let $Q(X,Y,Z)$ be a homogeneous quadratic equation defining $\cC$.
Since the characteristic of $F$ is odd by assumption, by the theory
of quadratic forms \cite[$\S$IV.1]{SerreCA} we can assume that $Q$
is diagonal and non-degenerate
\begin{equation}\label{eq-conic}
Q(X,Y,Z)=aX^2+bY^2+cZ^2, \quad a,b,c\in F^\times.
\end{equation}
Let $d:=abc\neq 0$ be the discriminant of $Q$. The conic $\cC$
obviously does not change if we replace $Q$ with $\alpha\cdot Q$ for
any $\alpha\in F^\times$. Next, $Q(X,Y,Z)$ and $Q(\alpha X, Y, Z)$
define isomorphic conics. Since $Q(dX,Y,Z)/d$ has discriminant $1$,
we can assume that $Q$ is given by (\ref{eq-conic}) and has
discriminant $1$. Let $\eps_v(Q)=\pm 1$ be the \textit{Hasse
invariant} of $Q$ at $v\in |C|$. The argument in the proof of
\cite[Thm. 6, p. 36]{SerreCA} shows that $Q$ has a non-trivial zero
over $F_v$ if and only if $\eps_v(Q)=(-1,-d)_v$. On the other hand,
$Q$ having non-trivial zeros over $F_v$ is equivalent to
$\cC(F_v)\neq \emptyset$. By the Hasse-Minkowski Theorem \cite[p.
189]{OMeara}, a non-degenerate quadratic form over $F$, up to a
linear isomorphism, is uniquely determined by the number of
variables, the discriminant and the local Hasse invariants for all
places of $F$. We conclude that two conics $\cC$ and $\cC'$ in
$\P^2_F$ are isomorphic over $F$ if and only if
$$\cC(F_v)\neq \emptyset \Longleftrightarrow \cC'(F_v)\neq \emptyset
\quad\text{ for all } v\in |C|.$$

The projective curve $X^\cD_F$ is defined over $F$ and is smooth,
cf. Theorem \ref{thmMC}. The existence of the uniformization
(\ref{eq-unifinf}) implies that $X^\cD_F$ is geometrically
connected. The genus of $X^\cD_F$ is $g(R)$, so if $R=\{x, y\}$ and
$\deg(x)=\deg(y)=1$, then $X^\cD_F$ has genus $0$, i.e., is a conic.
If $v\in |C|-R-\infty$, then by Theorem \ref{thmMC}, $X^\cD$ has
good reduction at $v$. A curve of genus zero over a finite field
always has rational points, so $X^\cD(F_v) \supset X^\cD(\F_v)\neq
\emptyset$ for $v \in |C|-R-\infty$. The same conclusion for
$v=\infty$ follows from Theorem \ref{thmLPinf}, as $\Odd(R)=1$.
Finally, by Theorem \ref{thm3.1}, $X^\cD(F_v)=\emptyset$ for $v\in
R$.

Now to prove the theorem it is enough to show that the conic defined
by $X^2-\xi Y^2 - \fr Z^2$ has $F_v$-rational points exactly for
$v\in |C|-R$. By the definition of the Hilbert symbol, $X^2 -\xi Y^2
-\fr Z^2$ has a non-trivial zero over $F_v$ if and only if $(\xi,
\fr)_v=1$. Thus, we need to show that $(\xi, \fr)_v=-1$ if and only
if $v \in R$. This calculation was already carried out in the proof
of Lemma \ref{lem6.1}.
\end{proof}


\section{On the Hasse principle}\label{SHasse} In this section $X$ is a smooth, projective,
geometrically irreducible curve over a field $K$.

\begin{defn} If $K$ is a global field,
denote the set of places of $K$ by $|K|$. \textit{$X$ violates the
Hasse principle} if $X(K_v)\neq \emptyset$ for all $v\in |K|$, but
$X(K)=\emptyset$.
\end{defn}

The study of varieties over global fields which violate the Hasse
principle is an active research area in Number Theory, cf.
\cite{Skorobogatov}. In this section we show that there exist
quadratic extensions of $F$ over which $X^\cD$ violates the Hasse
principle if $\deg(\fr)\geq 20$.

\begin{defn} The \textit{$K$-gonality} of $X$, denoted $\delta_K(X)$,
is the least positive integer $n$ for which there exists a degree
$n$ morphism $\pi:X\to \P^1_K$ defined over $K$.
\end{defn}

\begin{thm}\label{HPthm2}
Let $K$ be a finite field extension of $F$. Assume that the Jacobian
of $X$ has no isotrivial quotients. If $X$ has infinitely many
points of degree $d$ over $K$, then there exists a finite, purely
inseparable extension $\tilde{K}$ of $K$ such that
$\delta_{\tilde{K}}(X)\leq 2d$.
\end{thm}
\begin{proof}
This result is essentially due to Frey \cite[Prop. 2]{Frey}. The
statement of the theorem is proven in \cite[Thm. 2.1]{SchweizerMZ}
under the additional assumption that $X(K)\neq \emptyset$; the main
difference between \cite{SchweizerMZ} and the argument in
\cite{Frey} is that one needs to replace Faltings proof of
Mordell-Lang conjecture over number fields by its analogue over
function fields due to Hrushovski. Finally, as was observed by Clark
\cite[Thm. 5]{Clark}, the assumption $X(K)\neq \emptyset$ is not
necessary for Frey's argument to work.
\end{proof}

\begin{notn} Let $m(X)$ to be the minimum degree of a finite field extension
$L/K$ such that $X(L)\neq \emptyset$. If $K$ is a global field, let
$m_v(X):=m(X_{/K_v})$ and
$$
m_\loc(X)=\underset{v\in |K|}{\mathrm{lcm}} m_v(X).
$$
Note that due to the Weil bound and Hensel's lemma, $m_v(X)=1$ for
all but finitely many $v$, so $m_\loc$ is well-defined.
\end{notn}

\begin{thm}\label{HPthmClark} Let $K$ be a finite field extension of $F$. Suppose $X(K)=\emptyset$, and
moreover, suppose that for any finite, purely inseparable extension
$\tilde{K}$ of $K$ we have
$$
\delta_{\tilde{K}}(X)>2m>2
$$
for some multiple $m$ of $m_\loc(X)$. Then there exist infinitely
many extensions $L/K$ with $[L:K]=m$ such that $X_{L}$ violates the
Hasse principle.
\end{thm}
\begin{proof} Let $S_m$ be the set of points on $X$ of degree $m$
over $K$. Due to our assumption on the gonality of $X$, Theorem
\ref{HPthm2} implies that $S_m$ is finite. The rest of the proof is
the same as in \cite[Thm. 6]{Clark}.
\end{proof}

\begin{lem}\label{HPlem4}
Assume $K$ is a discrete valuation field with perfect residue field
$k$. If $X$ has good reduction $X_k$ over $k$, then $\delta_K(X)\geq
\delta_k(X_k)$.
\end{lem}
\begin{proof}
See Lemma 5.1 and Remark 5.2 in \cite{NS}.
\end{proof}

\begin{lem}\label{HPlem5}
Let $S\subset |C|$ be a finite subset of places. Denote
$s:=\sum_{x\in S}\deg(x)$. There exists a place $o\in |C|-S$ such
that $\deg(o)\leq \log_q(s)+1$.
\end{lem}
\begin{proof} Fix a natural number $n\geq 1$. Let $S_n=\{x\in |C|\ |\ \deg(x)\leq
n\}$. Then
$$
\sum_{x\in S_n}\deg(x)\geq \sum_{\substack{x\in |C|\\ \deg(x)|
n}}\deg(x)=\# C(\F_{q^n})=q^n+1.
$$

Choose $o\in |C|$ of least possible degree subject to $o\not\in S$.
A moment of thought shows that it is enough to prove the statement
of the lemma for $S=S_n$. In this case, the above inequality gives
$s> q^{n}$. Since we can choose $o$ of degree $n+1$, the claim
follows.
\end{proof}

\begin{prop}\label{HPprop6} The Jacobian of $X^\cD_F$ has no isotrivial quotients.
For any finite, purely inseparable extension $\tilde{F}$ of $F$ we
have
$$
\delta_{\tilde{F}}(X^\cD)\geq \frac{\prod_{x\in
R}(q_x-1)}{(q^2-1)(q\deg(\fr)+3)}>\frac{q^{\frac{\deg(\fr)}{2}-3}}{\deg(\fr)+3}.
$$
\end{prop}
\begin{proof}
The Jacobian $J^\cD$ of $X^\cD_F$ has no isotrivial quotients since
$X^\cD$ is totally degenerate at the places $R\cup \infty$, and
hence the connected component of the identity of the N\'eron model
of $J^\cD$ has purely toric reduction at those places.

Let $\tilde{F}$ be a finite, purely inseparable extension
$\tilde{F}$ of $F$. Fix some $o\in |C|-R-\infty$, and let
$\tilde{o}$ be the (unique) place of $\tilde{F}$ over $o$. As
$\tilde{F}/F$ is totally ramified, the residue field of the place
$\tilde{o}$ is $\F_o$. Denote $X^\cD_o:=X^\cD\times_C \Spec(\F_o)$.
By Lemma \ref{HPlem4}, $\delta:=\delta_{\F_o}(X^\cD_o)\leq
\delta_{\tilde{F}}(X^\cD)$, so it is enough to give a lower bound on
$\delta$. On the one hand, we must have
$$
\#X^\cD_o(\F_o^{(2)})\leq \delta\#\P^1_{\F_o}(\F_o^{(2)})
=\delta(q_o^2+1).
$$
On the other hand, by \cite[Cor. 4.8]{PapGenus}
$$
\#X^\cD_o(\F_o^{(2)})\geq \frac{1}{q^2-1}\prod_{x\in R\cup
o}(q_x-1)+\frac{q}{q+1}\cdot 2^{\# R}\cdot \Odd(R\cup o).
$$
We conclude that
$$
\frac{1}{q^2-1}\prod_{x\in R}(q_x-1)\leq
\delta\frac{q_o^2+1}{q_o-1}\leq \delta (q_o+3).
$$
By Lemma \ref{HPlem5}, we can choose $o$ such that $q_o\leq
q\deg(\fr)$. With this choice, we get the desired lower bound on
$\delta$. The second inequality follows from the crude estimates
$$
\prod_{x\in R}(q_x-1)\geq q^{\deg(\fr)/2}\quad \text{ and }\quad
(q^2-1)(q\deg(\fr)+3)< q^3(\deg(\fr)+3).
$$
\end{proof}

\begin{thm}\label{thmHP}
Suppose $\deg(\fr)\geq 20$. Then there exist infinitely many
quadratic extensions $L/F$ such that $X^\cD$ violates the Hasse
principle over $L$.
\end{thm}
\begin{proof} If $\deg(\fr)\geq 20$, then Proposition
\ref{HPprop6} gives
$$
\delta_{\tilde{F}}(X^\cD)>\frac{q^{\frac{\deg(\fr)}{2}-3}}{\deg(\fr)+3}>4.
$$
As a consequence of Theorems \ref{thm2.1}, \ref{thm3.1} and
\ref{thmLPinf}, we have $m_\loc(X^\cD)=2$ and $X^\cD(F)=\emptyset$.
The claim now follows from Theorem \ref{HPthmClark}.
\end{proof}


\subsection*{Acknowledgments}  I thank E.-U. Gekeler and A. Schweizer  for
useful discussions. The article was mostly written while I was
visiting the Department of Mathematics of Saarland University. I
thank the members of the department for their warm hospitality.



\end{document}